\documentclass{article}
\usepackage{graphicx,subfig} % Required for inserting images
\usepackage{amsmath,amsfonts,amssymb,amsthm}
\usepackage{wasysym}
\usepackage{xcolor}
\usepackage{verbatim}
\usepackage[title]{appendix}
\usepackage{paralist}
\usepackage{url}
\usepackage{mathtools}
\newtheorem{theorem}{Theorem}
\newtheorem{lemma}{Lemma}
\newtheorem{corollary}{Corollary}
\theoremstyle{definition}
\newtheorem{remark}{Remark}

\newtheorem{definition}{Definition}

\renewcommand{\le}{\leqslant}
\renewcommand{\ge}{\geqslant}
\renewcommand{\leq}{\leqslant}
\renewcommand{\geq}{\geqslant}
\renewcommand{\emptyset}{\varnothing}

\newcommand{\ca}{\mathcal{A}}
\newcommand{\tmod}{\ \mathsf{mod}\ }

\newcommand{\real}{\mathbb{R}}
\newcommand{\ints}{\mathbb{Z}}
\newcommand{\natu}{\mathbb{N}}

\newcommand{\bsa}{\boldsymbol{a}}

\newcommand{\bsd}{\boldsymbol{d}}
\newcommand{\bse}{\boldsymbol{e}}

\newcommand{\bsk}{\boldsymbol{k}}

\newcommand{\bsx}{\boldsymbol{x}}
\newcommand{\bsy}{\boldsymbol{y}}
\newcommand{\bst}{\boldsymbol{t}}

\newcommand{\bsz}{\boldsymbol{z}}
\newcommand{\bszero}{\boldsymbol{0}}
\newcommand{\bsone}{\boldsymbol{1}}
\newcommand{\bsell}{\boldsymbol{\ell}}
\newcommand{\bsalpha}{\boldsymbol{\alpha}}
\newcommand{\bscalE}{\boldsymbol{\mathcal{E}}}

\newcommand{\rd}{\,\mathrm{d}}

\newcommand{\dunif}{\mathbb{U}}

\newcommand{\e}{\mathbb{E}}
\newcommand{\var}{\mathrm{Var}}

\newcommand{\giv}{\!\mid\!}

\newcommand{\tran}{\mathsf{T}}

\newcommand{\walk}{\mathrm{wal}_k}
\newcommand{\walbsk}{\mathrm{wal}_{\bsk}}

\newcommand{\walkappa}[1]{{\mathrm{wal}_{#1}}}

\title{Automatic optimal-rate convergence of randomized nets using median-of-means}
\author{Zexin Pan\footnote{Johann Radon Institute for Computational and Applied Mathematics,
  \"OAW, Altenbergerstrasse~69, 4040~Linz, Austria. (\texttt{zexin.pan@oeaw.ac.at}).}}
  \date{}

\begin{document}

\maketitle

%    Abstract is required.
\begin{abstract}
We study the sample median of independently generated quasi-Monte Carlo estimators based on randomized digital nets and prove it approximates the target integral value at almost the optimal convergence rate for various function spaces. In contrast to previous methods, the algorithm does not require a priori knowledge of underlying function spaces or even an input of pre-designed $(t,m,s)$-digital nets, and is therefore easier to implement. This study provides further evidence that quasi-Monte Carlo estimators are heavy-tailed when applied to smooth integrands and taking the median can significantly improve the error by filtering out the outliers.
\end{abstract}

\section{Introduction}

Quasi-Monte Carlo (QMC) method \cite{dick:pill:2010} has been used extensively as an alternative to the usual Monte Carlo (MC) algorithm. Just like MC, QMC uses the average of $n$ function evaluations as an estimate for the target integral value. The places of sampling are, however, carefully designed to explore the underlying sampling space efficiently. As a result, QMC is less sensitive to the curse of dimensionality compared to classical quadrature rules while enjoying a faster error convergence rate compared to MC. 

Traditionally, the error bound is given by the Koksma-Hlawka inequality \cite{hick:2014}, which proves a $O(n^{-1+\epsilon})$ convergence rate for arbitrarily small $\epsilon>0$ when the integrand $f$ has a finite total variation in the sense of Hardy and Krause. Later \cite{rtms} introduces the method of scrambling, which randomizes the sampling points without breaking their design. Randomization boosts the convergence rate to $O(n^{-3/2+\epsilon})$ as measured by the rooted mean squared error (RMSE) when $f$ has continuous dominating mixed derivatives of order $1$.

An issue with the standard QMC design is the convergence rate cannot be further improved even if $f$ satisfies a stronger smoothness assumption \cite{superpolyone}. An exception is the lattice rules, for which point sets can be designed to obtain the optimal convergence rate for a range of smoothness conditions \cite{dick2022lattice}, but this often requires $f$ to satisfy a periodic boundary condition.

To improve the convergence rate, \cite{dick:2011} constructs the higher order scrambled digital nets, which achieve a better than $O(n^{-3/2})$ RMSE convergence rate for sufficiently smooth $f$. A drawback with this design is the smoothness parameter is needed as an input for the construction, and still the algorithm fails to converge at the optimal rate once the smoothness of $f$ exceeds the parameter we set.

This paper offers a solution using the median trick. The median of QMC estimators is first analyzed in \cite{superpolyone} and proven to converge at $O(n^{-c\log_2(n)})$ rate for any $c<3\log(2)/\pi^2\approx 0.21$ when $f$ is a one-dimensional analytic function over $[0,1]$. Later \cite{superpolymulti} generalizes this result to higher-dimensional $f$. Meanwhile, \cite{superpolyone} also proves a $O(n^{-\alpha+\epsilon})$ convergence rate when $f$ is $\alpha$-times continuously differentiable over $[0,1]$. While the rate is not optimal, it demonstrates a feature of automatic convergence rate speedup without the knowledge of the smoothness of $f$. In this paper, we will improve this result by both generalizing it to higher-dimensional $f$ and proving a $O(n^{-\alpha-1/2+\epsilon})$ almost optimal rate. Our analysis will be focused on base-2 scrambled digital nets, although the idea can be generalized to other numerical integration methods \cite{chen2025randomintegrationalgorithmhighdimensional, goda2024simpleuniversalalgorithmhighdimensional, Goda2024}. As an additional merit of the median trick,  we do not need to input a pre-designed digital net to implement the algorithm.

This paper is organized as follows. Section~\ref{sec:back} introduces necessary background knowledge and notation for later sections. Subsection~\ref{subsec:mse} in particular explains how the median trick boosts the convergence rate. Section~\ref{sec:rate} is devoted to the study of convergence rate of the median estimator constructed by randomized digital nets under various smoothness assumptions. Subsection~\ref{subsec:Valpha} studies integrands with bounded variation of order $\alpha$ used by \cite{dick:2011}. We are going to show the median estimator attains almost the same convergence rate as the higher order scrambled digital nets without knowing the smoothness parameter $\alpha$. The proof is rather short thanks to a key lemma proven in \cite{dick:2011}, and it illustrates the standard analysis scheme of the median estimator. Subsection~\ref{subsec:Vlambda} studies integrands whose dominating mixed derivatives of order $\alpha$ are continuous with bounded fractional variation used by \cite{dick:2008}. We will prove an almost $1/2$ convergence rate improvement over the optimal deterministic error rate achieved by the digital $(t,\alpha,s)$-sequence demonstrated in \cite{dick:2008}. Subsection~\ref{subsec:Hp} studies integrands whose dominating mixed derivatives belong to the fractional Sobolev space used by \cite{JOSEF:2008} and \cite{gnewuch2024qmcintegrationbasedarbitrary}. Again we will show an almost $1/2$ convergence rate improvement over the optimal deterministic rate. The proof resembles that of Subsection~\ref{subsec:Vlambda}, so we will omit the repeated part and highlight where a different analysis is required. Section~\ref{sec:exp} showcases many numerical experiments that test how well median estimators perform compared to higher order digital nets. Section~\ref{sec:disc} concludes the paper with a discussion of the limitations of our analysis as well as interesting research directions.

\section{Background and notation}\label{sec:back}

We use $\natu$ for the set of natural numbers $\{1,2,3,\dots\}$ and $\natu_0=\natu\cup\{0\}$. Because we often need to remove the zero frequency, we let $\natu^s_* = \natu_0^s\setminus\{\bszero\}$. We further define $\ints_{\le\ell}=\{0,1,\dots,\ell\}$ for $\ell\in \natu_0$ and $\ints_{<\ell}=\{0,1,\dots,\ell-1\}$ for $\ell\in \natu$. The dimension is denoted as $s$ in this paper. For $s\in \natu$, we write $1{:}s=\{1,2,\dots,s\}$. For a vector $\bsx\in\real^s$ and a subset $u\subseteq1{:}s$, we use
$\bsx_u$ to denote the $|u|$-dimensional vector of elements $x_j$ for $j\in u$ and $\bsx_{u^c}$ to denote the $(s-|u|)$-dimensional vector of elements $x_j$ for $j\not\in u$. We use $\bsone_{1{:}s}$ for the vector of $s$ ones and $\bsone\{\mathcal{A}\}$ for the indicator function that equals $1$ if event $\mathcal{A}$ occurs and equals $0$ if not. For $\bsell=(\ell_1,\dots,\ell_s)\in\natu^s_0$, we use $\Vert \bsell\Vert_1$ to denote $\sum_{j=1}^s \ell_j$.  The cardinality of a set $K$ is denoted by $|K|$. For positive sequence $a_n$ and $b_n$, we write $a_n=O(b_n)$ if there exists a constant $C<\infty$ such that $a_n\le Cb_n$ for all but finitely many $n\in \natu$.

The integrand $f$ is assumed to have domain $[0,1]^s$. The $L^p$-norm of $f$ over $[0,1]^s$ is denoted as $\Vert f\Vert_{L^p([0,1]^s)}$ and the spaces of $f$ with finite $L^p$-norm is denoted as $L^p([0,1]^s)$. We also use $C([0,1]^s)$ to denote the space of continuous functions over $[0,1]^s$ and $\Vert f\Vert_\infty=\max_{\bsx\in [0,1]^s}|f(\bsx)|$  for $f\in C([0,1]^s)$. We always assume our integrand $f$ belongs to $C([0,1]^s)$.   QMC methods aim to approximate
\begin{equation*}
    \mu=\int_{[0,1]^s}f(\bsx)\rd\bsx
\quad
\text{by}
\quad
\hat\mu=\frac1n\sum_{i=0}^{n-1}f(\bsx_i)
\end{equation*}
for points $\bsx_i\in[0,1]^s$. We will choose $\bsx_i$ according to the base-2 digital net construction described below.

\subsection{Digital nets and randomization}

For $m\in \natu$ and $i\in \ints_{<2^m}$, we represent the binary expansion $i=\sum_{\ell=1}^{m}i_\ell 2^{\ell-1}$ compactly by $\vec{i}=\vec{i}[m] = (i_1,\dots,i_m)^\tran\in\{0,1\}^m$. Similarly for $a\in[0,1)$ and a precision $E$, we truncate the binary expansion $a=\sum_{\ell=1}^\infty a_\ell 2^{-\ell}$ at the $E$'th digit and represent it by $\vec{a}=\vec{a}[E]=(a_1,\dots,a_E)^\tran\in\{0,1\}^E$. When $a$  admits two binary expansions, we always choose the one with finitely many $a_\ell=1$.

Our base-2 digital nets over $[0,1]^s$ are constructed using $s$ matrices $C_j\in\{0,1\}^{E\times m}$ for $m\in \natu$. The unrandomized digital nets are defined as
\begin{align}\label{eq:plainqmc}
\vec{x}_{ij} = C_j\vec{i} \ \tmod 2
\end{align}
for $i\in \ints_{<2^m}$ and $j\in1{:}s$. Here $\vec{x}_{ij}$ is the precision $E$ binary representation of $x_{ij}$ and $\bsx_i=(x_{i1},\dots,x_{is})$. The truncated digits beyond the $E$'th digit are assumed to be all $0$. For unrandomized digital nets we typically have $E\leq m$.

We will study two types of randomization. They both have the form
\begin{align}\label{eqn:xequalMCiplusD}
\vec{x}_{ij} = C_j\vec{i} + \vec{D}_j\ \tmod 2
\end{align}
for a random matrix $C_j\in \{0,1\}^{E\times m}$ and a random vector $\vec{D}_j\in\{0,1\}^E$ with precision $E\ge m$. The vectors $\vec{D}_j$ are called digital shifts and consist of independent $\dunif\{0,1\}$ random entries in both cases. The difference is in the randomization of $C_j$. 

In the first type of randomization, which we will call completely random designs, all entries of $C_j$ are drawn independently from $\dunif\{0,1\}$. It is very easy to implement, but may generate bad designs especially when $m$ is small. The second type of randomization comes from \cite{mato:1998:2} and is called random linear scrambling. It requires $s$ pre-designed generating matrices $\mathcal{C}_j\in \{0,1\}^{m\times m}$  as input and uses $C_j=M_j\mathcal{C}_j$ for a random lower-triangle matrix $M_j\in \{0,1\}^{E\times m}$. All entries on the diagonal of $M_j$ are $1$ and those below the diagonal are drawn independently from $\dunif\{0,1\}$. Typical choices of $\mathcal{C}_j$ are nets of Sobol' \cite{sobol67} and those of Niederreiter and Xing \cite{niedxing96} because they have low $t$ parameters that we will introduce below. Linear scrambling preserves the $t$ parameters and guarantees the resulting designs are as good as the unrandomized ones.

Letting $\bsx_i[E]$ denote the above points
constructed with precision $E$, our estimate is
\begin{equation}\label{eqn:muE}
   \hat\mu_E = \frac1n\sum_{i=0}^{n-1}f(\bsx_i)\quad\text{for $\bsx_i=\bsx_i[E]$}. 
\end{equation}
We will conveniently assume $E=\infty$ and focus our analysis on $\hat{\mu}_\infty$. In practice, $E$ is constrained by the floating point representation in use. Lemma 1 of \cite{superpolymulti} shows the difference between $\hat\mu_E$ and $\hat\mu_\infty$ is negligible when the error $|\hat{\mu}_\infty-\mu|$ is significantly larger than $\omega_f(\sqrt{s}2^{-E})$ where $\omega_f$ is the modulus of continuity of $f$. 
%The software by \cite{choi:etal:2021} allows one
%to specify $E$.

\subsection{Elementary intervals and the $t$ parameter}

An elementary interval in base $2$ is
a Cartesian product of the form
$$
\mathrm{EI}(\bsell,\bsa) = \prod_{j=1}^s
\Bigl[\frac{a_j}{2^{\ell_j}},
\frac{a_j+1}{2^{\ell_j}}
\Bigr)
$$
for $\bsell\in\natu_0^s$ and $\bsa\in \{(a_1,\dots,a_s)\mid a_j\in \ints_{<2^{\ell_j}}\}$. For $m\in\natu$, we say matrices $C_j,j\in 1{:}s$ generate a $(t,m,s)$-net in base $2$ if $\bsx_0,\dots,\bsx_{2^m-1}$ defined by equation~\eqref{eq:plainqmc} have the property that every elementary interval of volume $2^{t-m}$ contains precisely $2^t$ of those points. A lower value of $t$ implies better equidistribution of the point set. Sobol's construction has $t$ bounded in $m$ but growing at $O(s\log s)$ rate in $s$ \cite{sobol67}.

\subsection{Fourier-Walsh decomposition}

Walsh functions are the natural Fourier basis to study base-2 digital nets. For $k\in\natu_0$ and $x\in[0,1)$, the $k$'th univariate
Walsh function $\walk$ is defined by
$$
\walk(x) = (-1)^{\vec{k}^\tran\vec{x}}.
$$
Here $\vec{k}^\tran\vec{x}$ is computed using the infinite precision binary expansion of $k$ and $x$, but since $\vec{k}$ only contains finitely many nonzero entries, a finite precision representation suffices for the computation.

 For $\bsk \in\natu_0^s$,
$\walbsk:[0,1)^s\to\{-1,1\}$ is defined by the tensor product
$$
\walbsk(\bsx) =\prod_{j=1}^s\mathrm{wal}_{k_j}(x_j)
=(-1)^{\sum_{j=1}^s\vec{k}_j^\tran \vec{x}_j}.
$$

The functions $\walbsk$ for $\bsk\in\natu_0^s$
form a complete orthonormal basis of $L^2([0,1]^s)$~\cite{dick:pill:2010}, which implies
\begin{align}\label{eqn:Walshdecomposition}
f(\bsx) &= \sum_{\bsk\in\natu_0^s}\hat f(\bsk) \walbsk(\bsx),\quad\text{for}\\
\hat f(\bsk) &= \int_{[0,1]^s}f(\bsx)\walbsk(\bsx)\rd\bsx, \nonumber
\end{align}
where equality~\eqref{eqn:Walshdecomposition} holds in the $L^2$ sense. Using the above Walsh expansion, we can get the following error decomposition:
\begin{theorem}\label{thm:decomp}
Let $f\in C([0,1]^s)$  and let $\bsx_i$ be defined
by equation~\eqref{eqn:xequalMCiplusD} for $i\in\ints_{<2^m}$.
Then
\begin{equation}\label{eqn:errordecomposition}
    \hat{\mu}_{\infty}-\mu=\sum_{\bsk\in \natu_*^s}Z(\bsk)S(\bsk)\hat{f}(\bsk),
\end{equation}
where 
$$Z(\bsk)=\bsone\{\sum_{j=1}^s \vec{k{}}_j^\tran  C_j=\bszero \tmod 2\},$$
$$S(\bsk) = (-1)^{\sum_{j=1}^s\vec{k}_j^\tran \vec{D}_j}.$$
\end{theorem}
\begin{proof}
    Theorem 1 of \cite{superpolymulti} proves the statement under the stronger assumption $f$ is analytic, but the proof only needs equality~\eqref{eqn:Walshdecomposition} to hold in the pointwise sense, which is guaranteed by $f\in C([0,1]^s)$ and the arguments in Section 3.3 of \cite{dick:2008}.
\end{proof}

The set of $\bsk$ with $Z(\bsk)=1$ depends on the random matrices $C_j$ and therefore $\Pr(Z(\bsk)=1)$ depends on our randomization. This motivates the following definition:

\begin{definition}
    A randomization for $C_j\in \{0,1\}^{\infty\times m}, j\in 1{:}s$ is said to be \textbf{asymptotically full-rank} if there exists a constant $R$ independent of $m$ and $\bsk$ such that $\Pr(Z(\bsk)=1)\leq 2^{-m+R}$ for all $m\in \natu$ and $\bsk\in \natu^s_*$.
\end{definition}

\begin{lemma}\label{lem:PrZkbound}
When $E=\infty$, both completely random designs and random linear scrambling are asymptotically full-rank.
\end{lemma}
\begin{proof}
    For completely random designs, $\sum_{j=1}^s \vec{k{}}_j^\tran  C_j \tmod 2$ is uniformly distributed over $\{0,1\}^m$ when $\bsk\neq \bszero$ and $\Pr(Z(\bsk)=1)=2^{-m}$. For random linear scrambling, 
    Corollary 1 of \cite{superpolymulti} proves $\Pr(Z(\bsk)=1)\leq 2^{-m+t+s}$, so it satisfies the definition with $R=t+s$.
\end{proof}

\subsection{Mean squared error of the median}\label{subsec:mse}

Analysis of equation~\eqref{eqn:errordecomposition} is facilitated by the fact that $S(\bsk),\bsk\in \natu_*^s$ are mean 0 and pairwise independent (see Lemma 4 of \cite{superpolyone} for a proof when $s=1$). In particular, we can compute the variance of $\hat{\mu}_{\infty}$ and apply Chebyshev's inequality to bound the probability of large $|\hat{\mu}_{\infty}-\mu|$ occurrence. We can further sharpen the bound by identifying a small probability event $\ca$ such that 
$\e(\hat{\mu}_\infty\giv\ca^c)=\mu$ and 
$\var(\hat{\mu}_\infty\giv\ca^c)\ll\var(\hat{\mu}_\infty)$. We then turn this bound into a reduced RMSE by taking the median and filtering out the outliers. This idea leads to the following lemma:

\begin{lemma}\label{lem:medianMSE}
Let $\ca$ be an event with
$\Pr(\ca)\leq \delta$ and 
$\e(\hat{\mu}_\infty\giv \ca^c)=\mu$. Then the sample median 
$\hat{\mu}^{(r)}_{\infty}$ of  $2r-1$ independently generated values of 
$\hat{\mu}_\infty$ defined by equation~\eqref{eqn:muE} satisfies
\begin{align*}
    &\e((\hat{\mu}^{(r)}_{\infty}-\mu)^2)\leq \delta^{-1}{\Pr(\ca^c)}
\var(\hat{\mu}_\infty\giv \ca^c)+(8\delta)^r \Delta^2_m
\end{align*}
where $\Delta_m$ is the worst-case error of $\hat{\mu}_\infty$ when $n=2^m$ and $\Delta_m\leq 2\Vert f\Vert_\infty$ if $f\in C([0,1]^s)$. For random linear scrambling of base-2 $(t,m,s)$-nets, we further have $\Delta_m=O(m^{s-1}2^{-m})$ if $f$ has a finite total variation in the sense of Hardy and Krause.
\end{lemma}
\begin{proof}
    Lemma 5 of \cite{superpolymulti} proves the statement under the random linear scrambling, but the proof applies to any randomization on $C_j$ because the linear scrambling is assumed only to prove a tighter bound on $\Delta_m$. The bound $\Delta_m\leq 2\Vert f\Vert_\infty$ holds because both $\mu$ and $\hat{\mu}_\infty$ cannot exceed the maximum of $|f|$.
\end{proof}

As seen in Lemma~\ref{lem:PrZkbound} and Lemma~\ref{lem:medianMSE}, the two types of randomization differ only in the bounds of $\Pr(Z(\bsk)=1)$ and $\Delta_m$, so we will not emphasize which randomization is used when we state our convergence rate results in Section~\ref{sec:rate} with the understanding that the proof applies to any asymptotically full-rank randomization.

We will set $\mathcal{A}$ to be the event $\{Z(\bsk)=1$ for some $\bsk\in K\}$ for a chosen $K\subseteq \natu^s_*$. Because $Z(\bsk)$ is determined by $C_j$,  $j\in 1{:}s$ and $S(\bsk)$ is independent of those $C_j$, we know $S(\bsk)$ is independent of $\mathcal{A}$ and still has mean 0 conditioned on $\mathcal{A}^c$, and therefore $\e(\hat{\mu}_\infty\giv\ca^c)=\mu$. We also have $\Pr(Z(\bsk)=1\mid \ca^c)=0$ for $\bsk\in K$ and
\begin{align*}
    {\Pr(\ca^c)}
\var(\hat{\mu}_\infty\giv \ca^c)
&={\Pr(\ca^c)}\sum_{\bsk\in\natu^s_*} \Pr(Z(\bsk)=1\mid \ca^c)|\hat{f}(\bsk)|^2\\
&=\sum_{\bsk\in\natu^s_*\setminus K} \Pr(\{Z(\bsk)=1\}\cap \ca^c)|\hat{f}(\bsk)|^2\\
&\leq \sum_{\bsk\in\natu^s_*\setminus K} \Pr(Z(\bsk)=1) |\hat{f}(\bsk)|^2.
\end{align*}
A key ingredient in our later proof is to choose a proper set $K$ and make the above sum converge at a desired rate.

\section{Convergence rate of the median}\label{sec:rate}
Below we will study the convergence rate of the sample median 
$\hat{\mu}^{(r)}_{\infty}$ under different smoothness assumptions on $f$. In each subsection, we will first state the assumption and give related bounds of the Walsh coefficient $\hat{f}(\bsk)$. Then we will construct an event $\ca$ and use Lemma~\ref{lem:medianMSE} to prove the RMSE of $\hat{\mu}^{(r)}_{\infty}$ converges at almost the optimal rate. 

\subsection{When $f$ has bounded variation of order $\alpha$}

\label{subsec:Valpha}
In \cite{dick:2011}, the author constructs higher order scrambled digital nets and shows their RMSE achieves the optimal $O(n^{-\alpha-1/2+\epsilon})$ convergence rate for any $\epsilon>0$ when $f$ has bounded variation of order $\alpha$. We will show $\hat{\mu}^{(r)}_{\infty}$ automatically achieves the same convergence rate without knowing the smoothness parameter $\alpha$ as long as $r$ is large enough. 

First, we will define the variation of order $\alpha$ for $\alpha\in \natu$. We will mostly follow the notations from \cite{dick:2011}. For $f$ over $[0,1]$ and $\bsz=(z_1,\dots, z_\alpha)\in (-1,1)^\alpha$, we define the one-dimensional finite difference operator $\Delta_\alpha, \alpha\in \natu_0$ recursively as
$$\Delta_0(x)f=f(x),$$
$$\Delta_\alpha(x;z_1,\dots,z_\alpha)f=\Delta_{\alpha-1}(x+z_\alpha;z_1,\dots,z_{\alpha-1})f-\Delta_{\alpha-1}(x;z_1,\dots,z_{\alpha-1})f.$$
We assume $x+\sum_{r=1}^{\alpha'} z_r\in [0,1]$ for all $1\leq \alpha'\leq \alpha$.

For $f$ over $[0,1]^s$, we let $\Delta_{i,\alpha}$ be the finite difference operator applied to the $i$th coordinate of $f$. For $\bsalpha=(\alpha_1,\dots,\alpha_s)$ and $\bsz_j=(z_{j,1},\dots,z_{j,\alpha_j})$, we define
$$\Delta_{\bsalpha}(\bsx;\bsz_1,\dots,\bsz_s)f=\Delta_{s,\alpha_s}(x_s;\bsz_s)\cdots\Delta_{1,\alpha_1}(x_1;\bsz_1)f.$$

The generalized Vitali variation of $f$ with $\alpha\in\natu$ and $\bsalpha\in \{1,\dots,\alpha\}^s$ is defined as
\begin{equation}\label{eqn:Valphadef}
    V^{(s)}_{\bsalpha}(f)=\sup_{\mathcal{P_\alpha}}\Bigl(\sum_{J\in \mathcal{P_\alpha}}\mathrm{Vol}(J)\sup_{\bst,\bsz_1,\dots,\bsz_s}\Bigl|\frac{\Delta_{\bsalpha}(\bst;\bsz_1,\dots,\bsz_s)f}{\prod_{j=1}^s\prod_{r=1}^{\alpha_j} z_{j,r}}\Bigr|^2\Bigr)^{1/2}
\end{equation}
where $\mathrm{Vol}(J)$ is the Lebesgue measure of $J$, $\sup_{\mathcal{P_\alpha}}$ is over all partitions of $[0,1)^s$ into subcubes of the form $J=\prod_{j=1}^s [2^{-\alpha \ell_j}a_j,2^{-\alpha \ell_j}(a_j+1))$ for $a_j\in \ints_{<2^{\alpha\ell_j}}$ and $\ell_j\in \natu$, and $\sup_{\bst,\bsz_1,\dots,\bsz_s}$ is over $\bst\in J$ and $z_{j,r}=2^{-\alpha(\ell_j-1)-r}$ or $-2^{-\alpha(\ell_j-1)-r}$ such that $2^{-\alpha (\ell_j-1)}\lfloor 2^{-\alpha}a_j\rfloor\leq t_j+\sum_{r=1}^{\alpha'_j}z_{j,r}<2^{-\alpha (\ell_j-1)}(\lfloor 2^{-\alpha}a_j\rfloor+1)$ for all $1\leq \alpha'_j\leq \alpha_j$. We note that the definition used by \cite{dick:2011} allows partitions of more general shape, but it is always possible to find refinements from our set of partitions, so the two definitions agree.

The generalized Hardy and Krause variation of $f$ with $\alpha\in \natu$ is defined as
\begin{equation*}
    V_{\alpha}(f)=\Bigl(\sum_{u\subseteq 1{:}s}\sum_{\bsalpha\in \{1,\dots,\alpha\}^{|u|}}(V^{(|u|)}_{\bsalpha}(f_u))^2\Bigr)^{1/2}
\end{equation*}
where 
$f_u(\bsx_u)=\int_{[0,1]^{s-|u|}}f(\bsx)\rd \bsx_{1{:}s\setminus u}$ if $u\neq \emptyset$ and $V^{(\emptyset)}_{\bsalpha}(f_\emptyset)=|\mu|$.

Next for $d\in \natu$, we define the digit interlacing function 
$\mathcal{E}_d:\natu^d_0\to \natu_0$ as
    $$\mathcal{E}_d(k_1,\dots,k_d)=\sum_{\ell=1}^\infty \sum_{r=1}^d \vec{k}_r(\ell)2^{d (\ell-1)+r-1}$$
    where $\vec{k}_r(\ell)$ is the $\ell$'th entry of $\vec{k}_r$. The  vectorized version $\bscalE_d:\natu^{d s}_0\to \natu^s_0$ is given by
     $$\bscalE_d(k_1,\dots,k_{d s})=(\mathcal{E}_d(k_1,\dots,k_d),\dots,\mathcal{E}_d(k_{d(s-1)+1},\dots,k_{d s})).$$ 
     It is easy to see $\bscalE_d$ is a bijection between $\natu^{d s}_0$ and $\natu^s_0$.

With the above definition, \cite{dick:2011} proves the following bounds on the Walsh coefficients $\hat{f}(\bsk)$. We will prove a similar result in the next subsection under a different assumption.

\begin{lemma}\label{lem:Josefbound}
    Let $\alpha\in \natu$ and suppose $f$ over $[0,1]^s$ satisfies $V_{\alpha}(f)<\infty$. 
    For $\bsell=(\ell_1,\dots,\ell_{\alpha s})\in \natu_0^{\alpha s}$, define
    \begin{equation}\label{eqn:Bals}
        B_{\alpha,\bsell,s}=\{\bsk\in \natu_0^{\alpha s}\mid \lfloor 2^{\ell_j-1}\rfloor\leq k_j<2^{\ell_j}\text{ for }j\in 1{:}\alpha s\}.
    \end{equation}
    Then we have
    $$\Bigl(\sum_{\bsk\in B_{\alpha,\bsell,s}}|\hat{f}(\bscalE_\alpha(\bsk))|^2\Bigr)^{1/2}\leq C_{\alpha,s} 2^{-\alpha\Vert\bsell\Vert_1} V_{\alpha}(f)$$
    where $C_{\alpha,s}$ is an explicit constant depending on $\alpha$ and $s$ and $\Vert\bsell\Vert_1$ is the sum of the $\alpha s$ components of $\bsell$.
\end{lemma}

\begin{proof}
    This is Lemma 9 of \cite{dick:2011} with $d=\alpha$ and $b=2$. The constant $C_{\alpha,s}$ is given by
    $$C_{\alpha,s}=\sup_{\bsell\in \natu_0^{\alpha s}}2^{\alpha\Vert\bsell\Vert_1}\gamma(\bsell)=\prod_{i=1}^s \prod_{j=1}^{\alpha} 2^{-j+(i-1)\alpha+\alpha}=2^{s(s+1)\alpha^2/2-s\alpha(\alpha+1)/2}, $$
    where $\gamma(\bsell)$ is defined as in that lemma.
\end{proof}

Motivated by the above lemma, we define event $\mathcal{A}$ to be $\{Z(\bsk)=1$ for some $\bsk\in K\}$ with
$$K=\{\bsk\in \natu_*^s\mid \bsk=\bscalE_\alpha(\bsk'), \bsk' \in B_{\alpha,\bsell,s}, \Vert\bsell\Vert_1\leq m-\alpha s \log_2(m)\}.$$
The next lemma will help us bound the size of $|K|$.

\begin{lemma}\label{lem:geometricsum}
 For $\rho>1$ and any nondecreasing function $g:\natu_0\to \mathbb{R}_+$
$$\sum_{N=0}^{N_m} g(N) \rho^N\leq \frac{\rho}{\rho-1}g(N_m) \rho^{N_m}.$$
\end{lemma}
\begin{proof}
\begin{align*}
\sum_{N=0}^{N_m}g(N) \rho^{N}
         &= \sum_{N=0}^{N_m}g(N) \frac{\rho^{N+1}-\rho^{N}}{\rho-1}\\
        &=g(N_m) \frac{\rho^{N_m+1}}{\rho-1}-\sum_{N=1}^{N_m} (g(N)-g(N-1)) \frac{\rho^{N}}{\rho-1}-g(0)\frac{1}{\rho-1}\\
        &\leq \frac{\rho}{\rho-1}g(N_m) \rho^{N_m}.\qedhere
\end{align*}
\end{proof}

\begin{lemma}\label{lem:smallpVa}
    With $\mathcal{A}$ defined as above, we have $\Pr(\mathcal{A})=O(1/m)$.
\end{lemma}
\begin{proof}
    By Lemma~\ref{lem:PrZkbound} and a union bound argument, it suffices to show $|K|=O(2^m/m)$. Because $\bscalE_\alpha$ is a bijection and $B_{\alpha,\bsell,s}$ are disjoint for different values of $\bsell$, we need to show
    $$\sum_{\substack{\bsell\in \natu_0^{\alpha s}\\\Vert\bsell\Vert_1\leq m-\alpha s \log_2(m)}}|B_{\alpha,\bsell,s}|=O(\frac{2^m}{m}).$$
    From equation~\eqref{eqn:Bals}, we see that 
    $$|B_{\alpha,\bsell,s}|=\prod_{j=1}^{\alpha s} \max(2^{\ell_j-1},1)\leq 2^{\Vert\bsell\Vert_1}.$$
    Because $\{\bsell\in \natu_0^{\alpha s}\mid \Vert\bsell\Vert_1=N\}$ corresponds to the set of $\alpha s$ ordered nonnegative integers summing to $N$, 
\begin{equation}\label{eqn:setpartition}
    |\{\bsell\in \natu_0^{\alpha s}\mid \Vert\bsell\Vert_1=N\}|={N+\alpha s-1\choose \alpha s-1}.
\end{equation}
By applying Lemma~\ref{lem:geometricsum}
    \begin{align*}
       \sum_{\substack{\bsell\in \natu_0^{\alpha s}\\\Vert\bsell\Vert_1\leq m-\alpha s \log_2(m)}}|B_{\alpha,\bsell,s}|
        &\leq \sum_{N=0}^{\lfloor m-\alpha s \log_2(m)\rfloor}{N+\alpha s-1\choose \alpha s-1}2^{N}\\
        &\leq {m-\alpha s \log_2(m)+\alpha s-1\choose \alpha s-1}2^{m-\alpha s \log_2(m)+1}\\
        &=O(\frac{2^m}{m}),
    \end{align*}
    which implies our conclusion.
\end{proof}

\begin{theorem}\label{thm:rateVa}
    Let $\alpha\in \natu$ and suppose $V_{\alpha}(f)<\infty$. When $r\geq m$, we have $\e((\hat{\mu}^{(r)}_{\infty}-\mu)^2)=O(n^{-2\alpha-1+\epsilon})$ for any $\epsilon>0$.
\end{theorem}
\begin{proof}
    By Lemma~\ref{lem:smallpVa}, we can apply Lemma~\ref{lem:medianMSE} with $\delta=C/m$ for some large enough $C>0$. Then for large enough $m$, $8\delta<1$ and $(8\delta)^r\leq (8C/m)^m\to 0$ at a rate faster than $n^{-2\alpha-1}=2^{-(2\alpha+1)m}$ for any $\alpha\in\natu$. It remains to show
    $${\Pr(\ca^c)}
\var(\hat{\mu}_\infty\giv \ca^c)\leq \sum_{\bsk\in\natu_*^s\setminus K }  \Pr(Z(\bsk)=1)|\hat{f}(\bsk)|^2=O(n^{-2\alpha-1+\epsilon}).$$

By Lemma~\ref{lem:PrZkbound}, for some constant $R$
\begin{align*}
    &\sum_{\bsk\in\natu_*^s\setminus K }  \Pr(Z(\bsk)=1)|\hat{f}(\bsk)|^2\\
    \leq & 2^{-m+R}\sum_{\bsk\in\natu_*^s\setminus K} |\hat{f}(\bsk)|^2\\
    =& 2^{-m+R}\sum_{\substack{\bsell\in \natu_0^{\alpha s}\\\Vert\bsell\Vert_1> m-\alpha s \log_2(m)}} \sum_{\bsk\in B_{\alpha,\bsell,s}}|\hat{f}(\bscalE_\alpha(\bsk))|^2\\
    \leq&  2^{-m+R}\sum_{\substack{\bsell\in \natu_0^{\alpha s}\\\Vert\bsell\Vert_1> m-\alpha s \log_2(m)}}  C^2_{\alpha,s} 4^{-\alpha\Vert\bsell\Vert_1} V_{\alpha}(f)^2,
\end{align*}
where we applied Lemma~\ref{lem:Josefbound} in the last step.

Next we use equation~\eqref{eqn:setpartition} and get
\begin{align*}
\sum_{\substack{\bsell\in \natu_0^{\alpha s}\\\Vert\bsell\Vert_1> m-\alpha s \log_2(m)}}  4^{-\alpha\Vert\bsell\Vert_1} &=\sum_{N=\lceil m-\alpha s \log_2(m) \rceil}^\infty {N+\alpha s-1\choose \alpha s-1} 4^{-\alpha N}\\
& \leq \sum_{N=\lceil m-\alpha s \log_2(m) \rceil}^\infty \frac{(N+\alpha s-1)^{\alpha s-1}}{(\alpha s-1)!}  4^{-\alpha N}.
\end{align*}
For arbitrarily small $\epsilon>0$, we have  $(N+\alpha s-1)^{\alpha s-1}<4^{\epsilon N}$ for large enough $N$. Therefore for large enough $m$
\begin{align}\label{eqn:epilsonrate}
\sum_{\substack{\bsell\in \natu_0^{\alpha s}\\\Vert\bsell\Vert_1> m-\alpha s \log_2(m)}}  4^{-\alpha\Vert\bsell\Vert_1}
&\leq \frac{1}{(\alpha s-1)!} \sum_{N=\lceil m-\alpha s \log_2(m) \rceil}^\infty  4^{-(\alpha-\epsilon) N}\nonumber\\
&=\frac{4^{-(\alpha-\epsilon)\lceil m-\alpha s \log_2(m) \rceil}}{(\alpha s-1)!(1-4^{-\alpha+\epsilon})}.
\end{align}
Finally we have for large enough $m$
  $${\Pr(\ca^c)}
\var(\hat{\mu}_\infty\giv \ca^c)\leq 2^{-m+R} C^2_{\alpha,s}V_{\alpha}(f)^2  \frac{4^{-(\alpha-\epsilon)\lceil m-\alpha s \log_2(m) \rceil}}{(\alpha s-1)!(1-2^{-\alpha+\epsilon})}.$$
Because $n=2^m$ and $\log_2(m)<\epsilon m$ for large enough $m$, our proof is complete once we choose a small enough $\epsilon$.
\end{proof}

\subsection{When dominating mixed derivatives of $f$ up to order $\alpha$ are continuous with bounded fractional variation} \label{subsec:Vlambda}

Next we consider the case when $f$ has dominating mixed derivatives up to order $\alpha$, namely $f$ can be partially differentiated in each variable up to $\alpha$ times. We will use the notation
$$f^{(\bsalpha)}=f^{(\alpha_1,\dots,\alpha_s)}=\frac{\partial^{\sum_{j=1}^s\alpha_j} f}{\partial x_1^{\alpha_1}\cdots\partial x_s^{\alpha_s}}.$$
We conventionally define $f^{(\bszero)}=f$. We say $f\in C^{(\alpha,\dots,\alpha)}([0,1]^s)$ if $f^{(\bsalpha)}\in C([0,1]^s)$ for all $\bsalpha\in \ints^s_{\le \alpha}$.

Notice that if $f^{(\bsalpha)}$ is continuous, equation~\eqref{eqn:Valphadef} can be viewed as a Riemann sum for the integral
$$\Bigl(\int_{[0,1]^s}f^{(\bsalpha)}(\bsx)^2\rd\bsx\Bigr)^{1/2}$$
and we immediately have $V^{(s)}_{\bsalpha}(f)<\infty$. Hence $f\in C^{(\alpha,\dots,\alpha)}([0,1]^s)$ for $\alpha\in \natu$ implies $V_\alpha(f)<\infty$ and the results from the previous subsection applies here as well.

Next we introduce the fractional Vitali variation with $\mathfrak{p}=2$ from \cite{dick:2008}. For a cube $J=\prod_{j=1}^s [d_j,e_j)$, we let $\Delta(f,J)$ to be the alternating sum of $f$ at vertices of $J$ where adjacent vertices have opposite signs. In other words, if we use $(\bsd_u,\bse_{u^c})$ to denote the $\bsx$ with $x_j=d_j$ if $j\in u$ and $x_j=e_j$ if $j\notin u$, then 
\begin{equation}\label{eqn:deltaf}
\Delta(f,J)=\sum_{u\subseteq 1{:}s}(-1)^{s-|u|}f\Big((\bsd_u,\bse_{u^c})\Big).
\end{equation}

For $0<\lambda\leq 1$, the fractional Vitali variation of order $\lambda$ is defined as
\begin{equation*}
    V^{(s)}_\lambda(f)=\sup_{\mathcal{P}}\Bigl(\sum_{J\in \mathcal{P}}\mathrm{Vol}(J)\Bigl|\frac{\Delta(f,J)}{\mathrm{Vol}(J)^\lambda}\Bigr|^2\Bigr)^{1/2}
\end{equation*}
where $\sup_{\mathcal{P}}$ is over all partitions of $[0,1)^s$ into subcubes $J$ of forms $\prod_{j=1}^s[a_j,b_j)$. For instance, a one-dimensional Brownian motion path over $[0,1]$ almost surely has a finite variation of order $\lambda$ for any $\lambda<1/2$ because it is almost surely $\lambda-$Hölder continuous for $\lambda<1/2$ \cite{Durrett2010}.

We also need a notion of variation with respect to a margin. For this purpose, we define, for $u\subseteq 1{:}s$, the partial integral operator $I_u:L^1([0,1]^s)\to L^1([0,1]^{|u|})$ as
\begin{equation}\label{eqn:Iu}
    I_u(f)(\bsx_u)=\int_{[0,1]^{s-|u|}}f(\bsx_{u^c};\bsx_u)\rd\bsx_{u^c}.
\end{equation}
When $u=1{:}s$, $I_{1{:}s}$ simply maps $f$ to itself. We also define 
\begin{equation}\label{eqn:Du}
    D_{u^c}=\{\rho_{u^c}\mid \rho_{u^c}(\bsx)=\prod_{j\in u^c}\rho_j(x_j), \rho_j\in C([0,1]), \rho_j\geq 0, \Vert \rho_j\Vert_\infty\leq 1\},
\end{equation}
namely functions that are products of one-dimensional nonnegative continuous functions $\rho_j$ in variable $x_j$ with $\Vert \rho_j\Vert_\infty\leq 1$ for $j\in u^c$. The fractional Vitali variation of order $\lambda$ with respect to a nonempty subset $u\subseteq 1{:}s$ is defined as
\begin{equation}\label{eqn:fractionalvariation}
   V^{u}_\lambda(f)=\sup_{\rho_{u^c}\in D_{u^c}} V^{(|u|)}_\lambda \Big(I_u(f\rho_{u^c})\Big). 
\end{equation}
When $u=1{:}s$, we conventionally define $V^{1{:}s}_\lambda(f)=V^{(s)}_\lambda(f)$. We note that because $\Vert \rho_j\Vert_\infty \leq 1$, for any cube $J\subseteq [0,1]^{|u|}$,
\begin{align*}
    \Big|\Delta(I_u(f\rho_{u^c}),J)\Big|^2
    &=\Bigl|\int_{[0,1]^{s-|u|}}\Delta( f( \text{ }\cdot\text{ };\bsx_{u^c}),J)\prod_{j\in u^c}\rho_j(x_j)\rd\bsx_{u^c}\Big|^2\\
    &\leq \int_{[0,1]^{s-|u|}}\Big|\Delta( f( \text{ }\cdot\text{ };\bsx_{u^c}),J)\Big|^2\rd\bsx_{u^c}.
\end{align*}
Hence
\begin{align*}
    \Big(V^{(|u|)}_\lambda \Big(I_u(f\rho_{u^c})\Big)\Big)^2
    & = \sup_{\mathcal{P}}\sum_{J\in \mathcal{P}}\mathrm{Vol}(J)\Bigl|\frac{\Delta(I_u(f\rho_{u^c}),J)}{\mathrm{Vol}(J)^\lambda}\Bigr|^2\\
    &\leq\sup_{\mathcal{P}}\int_{[0,1]^{s-|u|}}\sum_{J\in \mathcal{P}}\mathrm{Vol}(J)\Bigl|\frac{\Delta( f( \text{ }\cdot\text{ };\bsx_{u^c}),J)}{\mathrm{Vol}(J)^\lambda}\Bigr|^2\rd\bsx_{u^c}\\
    &\leq \int_{[0,1]^{s-|u|}}\Bigl(V^{(|u|)}_\lambda\Bigl( f(\text{ }\cdot\text{ };\bsx_{u^c})\Bigr)\Bigr)^2\rd\bsx_{u^c}
\end{align*}
where the last inequality follows by taking the supremum over $\mathcal{P}$ inside the integral. Therefore, we know $V^{u}_\lambda(f)<\infty$ if every section $f(\text{ }\cdot\text{ };\bsx_{u^c})$ has a bounded fractional Vitali variation.

Finally for $f\in C^{(\alpha,\dots,\alpha)}([0,1]^s)$ and $0<\lambda\leq 1$, we define 
$$N_{\alpha,\lambda}(f)=\sup_{\substack{u\subseteq 1{:}s\\u\neq \emptyset}}\Bigl(\sum_{\bsalpha\in \ints_{\le\alpha}^s}\Bigl(V^{u}_\lambda(f^{(\bsalpha)})\Bigr)^2\Bigr)^{1/2}.$$
Our goal in this subsection is to show the RMSE of $\hat{\mu}^{(r)}_\infty$ has a $O(n^{-\alpha-\lambda-1/2+\epsilon})$ convergence rate when $N_{\alpha,\lambda}(f)<\infty$. This rate is almost optimal because Theorem 5 of \cite{hein:nova:2002} says the optimal rate is of order $n^{-\alpha-\lambda-1/2}$ for the Hölder class $F^{\alpha,\lambda}_1$, which is a subspace of functions with finite $N_{\alpha,\lambda}(f)$ when $s=1$.

A crucial step in our proof is to bound $|\hat{f}(\bsk)|$ when $N_{\alpha,\lambda}(f)<\infty$. We are going to bound $|\hat{f}(\bsk)|$ differently according to the number of nonzeros in the binary representation of $\bsk$, so we introduce the following notations. For $k = \sum_{\ell=1}^\infty a_\ell 2^{\ell-1}$, we define the set of nonzero bits $\kappa = \{\ell\in \natu\mid a_\ell=1\}$. We will use $k$ and $\kappa$ interchangeably since one specifies the other. The number of nonzero bits of $k$ is given by the cardinality of $\kappa$ denoted as $|\kappa|$. We further use $\lceil\kappa\rceil_{q}$ to denote the $q$'th largest element of $\kappa$ and $\lceil\kappa\rceil_{1{:}q}$ to denote the subset of $\kappa$ containing the $q$ largest elements if $q\le |\kappa|$. We conventionally define $\lceil\kappa\rceil_{1{:}q}=\emptyset$ if $q=0$ and $\lceil\kappa\rceil_{q}=0$ if $q>|\kappa|$. Finally we let $\lceil\bsk\rceil_{q}=(\lceil\kappa_1\rceil_{q},\dots,\lceil\kappa_s\rceil_{q})$.

The following lemma relates $\hat{f}(\bsk)$ to the derivatives of $f$.

\begin{lemma}\label{lem:exactfk}
Let $\bsalpha\in \natu^s_0$ and $f^{(\bsalpha)}\in C([0,1]^s)$. If $ |\kappa_j|\geq \alpha_j$ for all $j\in 1{:}s$,
\begin{equation}\label{eqn:exactfk}
\hat{f}(\bsk)=(-1)^{\sum_{j=1}^s\alpha_j}\int_{[0,1]^s}f^{(\bsalpha)}(\bsx)\prod_{j=1}^s \walkappa{\kappa_j\setminus \lceil\kappa_j\rceil_{1{:}\alpha_j}}(x_j)W_{\lceil\kappa_j\rceil_{1{:}\alpha_j}}(x_j)\rd \bsx,
\end{equation} 
where 
$$\walkappa{\kappa}(x)=(-1)^{\vec{k}^\tran \vec{x}}=\prod_{\ell\in \kappa}(-1)^{\vec{x}(\ell)}$$
with $\vec{x}(\ell)$ denoting the $\ell$'th component of $\vec{x}$ and $W_\kappa (x)$ is defined recursively as
$$W_\emptyset(x)=1,$$
\begin{equation}\label{eqn:Wkdef}
  W_{\kappa}(x)=\int_{[0,1)}(-1)^{\vec{x}(\lfloor\kappa\rfloor)}W_{\kappa\setminus \lfloor\kappa\rfloor}(x)\rd x  
\end{equation}
where $\lfloor\kappa\rfloor$ denotes the smallest element of $\kappa$. In particular, for nonempty $\kappa$ we have $W_\kappa(x)$ is continuous, nonnegative, periodic with period $2^{-\lfloor\kappa\rfloor+1}$ and
\begin{equation}\label{eqn:Wint}
\int_{[0,1]}W_\kappa(x)\rd x=\prod_{\ell\in\kappa}2^{-\ell-1},
\end{equation}
\begin{equation}\label{eqn:Wmax}
    \max_{x\in [0,1]} W_\kappa(x)= 2\prod_{\ell\in\kappa} 2^{-\ell-1}.
\end{equation}
\end{lemma}
\begin{proof}
Equation~\eqref{eqn:exactfk} comes from Theorem 2.5 of \cite{SUZUKI20161} while properties of $W_\kappa(x)$ is proven in Section 3 of \cite{SUZUKI20161}.
\end{proof}

Because we will often take the maximum and minimum over $j\in 1{:}s$, we will omit the argument $j\in 1{:}s$ when it is clear from the context. For instance, we will denote the maximum and minimum of $|\kappa_j|$ over $j\in 1{:}s$ as $\max |\kappa_j|$ and $\min |\kappa_j|$. 

First we derive a crude bound on $|\hat{f}(\bsk)|$ that applies to all $\bsk\neq \bszero$.

\begin{lemma}\label{lem:fkboundlealpha}
Let $f\in C^{(1,\dots,1)}([0,1]^s)$. If $\bsk\neq \bszero$, then
$$|\hat{f}(\bsk)|\leq C_f 2^{-\max \lceil \kappa_j\rceil_1} $$
where 
$$C_f=\max\Bigl\Vert\frac{\partial f}{\partial x_j}\Bigr\Vert_{L^1([0,1]^s)}<\infty.$$
\end{lemma}
\begin{proof}
    $f\in C^{(1,\dots,1)}([0,1]^s)$ implies $\partial f/{\partial x_j}\in C([0,1]^s)$ for each $j\in 1{:}s$ and hence $C_f<\infty$. Let $J$ be any index $j$ that maximizes $\lceil \kappa_j\rceil_1$. By applying Lemma~\ref{lem:exactfk} with $\bsalpha=(\alpha_1,\dots,\alpha_s)$ where $\alpha_j=1$ if $j=J$ and $\alpha_j=0$ if $j\neq J$, we have 
\begin{align*}
    |\hat{f}(\bsk)|&=\Big|\int_{[0,1]^s}\frac{\partial f}{\partial x_J}(\bsx)\Big(\prod_{j=1}^s \walkappa{\kappa_j\setminus \lceil\kappa_j\rceil_{1{:}\alpha_j}}(x_j)\Big)W_{\{\lceil\kappa_J\rceil_{1}\}}(x_J)\rd \bsx\Big|\\
    &\leq \Big(\max_{x\in [0,1]} W_{\{\lceil\kappa_J\rceil_{1}\}}(x)\Big)\int_{[0,1]^s}\Big|\frac{\partial f}{\partial x_J}(\bsx)\Big|\rd \bsx\\
    &=2^{-\lceil\kappa_J\rceil_{1}}\Bigl\Vert\frac{\partial f}{\partial x_j}\Bigr\Vert_{L^1([0,1]^s)},
\end{align*}
where we have used $|\walkappa{\kappa}(x)|=1$ and equation~\eqref{eqn:Wmax}. The conclusion follows once we maximize the bound over $j\in 1{:}s$.
\end{proof}

We need a finer bound when $ \lceil \bsk\rceil_{\alpha+1}\neq \bszero$. First we study the case when all components of $\lceil \bsk\rceil_{\alpha+1}$ are positive.

\begin{theorem}\label{thm:fkboundgealpha}
Let $\alpha\in \natu_0$ and $0<\lambda\le 1$. Suppose $f^{(\alpha,\dots,\alpha)}\in C([0,1]^s)$ and $V^{(s)}_\lambda(f^{(\alpha,\dots,\alpha)})<\infty$. For $\bsell=(\ell_1,\dots,\ell_s)\in \natu^s$, we define
$$B_{\alpha,\bsell,s}=\{\bsk\in \natu_*^s\mid  \lceil\bsk\rceil_{\alpha+1}=\bsell\}.$$
 Then
 \begin{equation*}
    \Bigl(\sum_{\bsk\in B_{\alpha,\bsell,s}}|\hat{f}(\bsk)|^2\Bigr)^{1/2}\leq V^{(s)}_\lambda(f^{(\alpha,\dots,\alpha)}) 2^{-(\alpha+\lambda)\Vert\bsell\Vert_1}.
 \end{equation*}
\end{theorem}
\begin{proof} Note that $\lceil\bsk\rceil_{\alpha+1}=\bsell$ implies $\min |\kappa_j|>\alpha$ because $\min\ell_j>0$. For simplicity, we denote $\lceil\kappa_j\rceil_{1{:}\alpha}$ as $\kappa^+_j$ and $\kappa_j\setminus\lceil\kappa_j\rceil_{1{:}\alpha+1}$ as $\kappa^-_j$. For each $\bsk\in B_{\alpha,\bsell,s}$, $\kappa_j$ is the disjoint union of $\kappa^+_j,\kappa^-_j$ and $\ell_j$. $\kappa^+_j$ comprises of $\alpha$ distinct numbers larger than $\ell_j$. $\kappa^-_j$ is an arbitrary subset of $1{:}\ell_j-1$, where $1{:}\ell_j-1=\emptyset$ if $\ell_j=1$.

To simplify the proof, we let $g(\bsx)=f^{(\alpha,\dots,\alpha)}(\bsx)$. Equation~\eqref{eqn:exactfk} with $\bsalpha=(\alpha,\dots,\alpha)$ for $\bsk\in B_{\alpha,\bsell,s}$ can be written as
\begin{align*}
|\hat{f}(\bsk)|^2
&=\Bigl|\int_{[0,1]^s}g(\bsx)\prod_{j=1}^s\walkappa{\kappa^-_j}(x_j)(-1)^{\vec{x}_j(\ell_j)}W_{\kappa^+_j}(x_j)\rd \bsx\Bigr|^2\\
&=\Bigl|\sum_{\bsa\in A}\int_{\mathrm{EI}(\bsell-\bsone_{1{:}s},\bsa)}g(\bsx)\prod_{j=1}^s\walkappa{\kappa^-_j}(x_j)(-1)^{\vec{x}_j(\ell_j)}W_{\kappa^+_j}(x_j)\rd \bsx\Bigr|^2
\end{align*}
where 
$$\mathrm{EI}(\bsell-\bsone_{1{:}s},\bsa)=\prod_{j=1}^s [\frac{a_j}{2^{\ell_j-1}},\frac{a_j+1}{2^{\ell_j-1}}),$$
$$A=\{(a_1,\dots,a_s)\mid a_j\in\ints_{<2^{\ell_j-1}}\}.$$
Because $\walkappa{\kappa^-_j}(x_j)$ is constant over $x_j\in[\frac{a_j}{2^{\ell_j-1}},\frac{a_j+1}{2^{\ell_j-1}})$ and $W_{\kappa^+_j}(x_j)$ has a period that divides $2^{-\ell_j}$, we can use $\Delta(f,J)$ defined in equation~\eqref{eqn:deltaf} and write
\begin{align}\label{eqn:fktodeltaf}
&\int_{\mathrm{EI}(\bsell-\bsone_{1{:}s},\bsa)}g(\bsx)\prod_{j=1}^s\walkappa{\kappa^-_j}(x_j)(-1)^{\vec{x}_j[\ell_j]}W_{\kappa^+_j}(x_j)\rd \bsx\nonumber\\
=&\Big(\prod_{j=1}^s\walkappa{\kappa^-_j}(\frac{a_j}{2^{\ell_j-1}})\Big) \underbrace{\int_{\mathrm{EI}(\bsell,2\bsa)} \Delta(g,\bsx+J_{\bsell})\prod_{j=1}^s W_{\kappa^+_j}(x_j)\rd \bsx}_{f_{\bsa}}
\end{align}
where
$$\mathrm{EI}(\bsell,2\bsa)=\prod_{j=1}^s [\frac{2a_j}{2^{\ell_j}},\frac{2a_j+1}{2^{\ell_j}}),$$
$$\bsx+J_{\bsell}=\prod_{j=1}^s [x_j,x_j+2^{-\ell_j}).$$
Next we notice that just as $\{\walbsk(\bsx)\mid\bsk\in\natu_0^s\}$ form an orthonormal basis of $L^2([0,1]^s)$, 
$$\Bigl\{\frac{1}{\sqrt{|A|}}\prod_{j=1}^s\walkappa{\kappa^-_j}(\frac{a_j}{2^{\ell_j-1}})\Bigl | \kappa^-_j\subseteq 1{:}\ell_j-1,j\in 1{:}s\Bigr\}$$
form an orthonormal basis of the $|A|$-dimensional vector space indexed by $\bsa\in A$. Hence once we sum $|\hat{f}(\bsk)|^2$ over all possible $\kappa^-_j\subseteq 1{:}\ell_j-1,j\in 1{:}s$ while fixing $\kappa^+_j, j\in 1{:}s$, we have by equation~\eqref{eqn:fktodeltaf} and Parseval's identity
\begin{align}\label{eqn:fksquaresum}
\sum_{\kappa^-_1,\dots,\kappa^-_s}|\hat{f}(\bsk)|^2=\sum_{\kappa^-_1,\dots,\kappa^-_s}\Bigl|\sum_{\bsa\in A}\Bigl(\frac{1}{\sqrt{|A|}}\prod_{j=1}^s\walkappa{\kappa^-_j}(\frac{a_j}{2^{\ell_j-1}})\Bigr)\sqrt{|A|}f_{\bsa}\Bigr|^2
=|A|\sum_{\bsa\in A}|f_{\bsa}|^2.
\end{align}
Next, Hölder's inequality implies
\begin{equation}\label{eqn:supdelta}
 |f_{\bsa}|\leq \Big(\sup_{\bsx\in \mathrm{EI}(\bsell,2\bsa)}|\Delta(g,\bsx+J_{\bsell})|\Big)\Bigl(\int_{\mathrm{EI}(\bsell,2\bsa)} \prod_{j=1}^s W_{\kappa^+_j}(x_j)\rd \bsx\Bigr). 
\end{equation}
Using periodicity of $W_{\kappa}(x)$, we have
\begin{align*}
\int_{\mathrm{EI}(\bsell,2\bsa)} \prod_{j=1}^s W_{\kappa^+_j}(x_j)\rd \bsx
&=\prod_{j=1}^s \int_{[2a_j2^{-\ell_j},(2a_j+1)2^{-\ell_j})}W_{\kappa^+_j}(x_j)\rd x_j\\
&=\prod_{j=1}^s 2^{-\ell_j}\int_{[0,1)}W_{\kappa^+_j}(x_j)\rd x_j\\
&=\prod_{j=1}^s 2^{-\ell_j}\prod_{\ell'\in \kappa^+_j}2^{-\ell'-1}
\end{align*}
where we applied equation~\eqref{eqn:Wint} in the last step. Putting $|A|=\prod_{j=1}^s 2^{\ell_j-1},\mathrm{Vol}(J_\ell)=\prod_{j=1}^s 2^{-\ell_j}$, equation~\eqref{eqn:supdelta} and the above equation into equation~\eqref{eqn:fksquaresum}, we get
\begin{align*}
\sum_{\kappa^-_1,\dots,\kappa^-_s}|\hat{f}(\bsk)|^2
&\leq \Bigl(\prod_{j=1}^s 2^{\ell_j-1}\Bigr)\sum_{\bsa\in A}\Big(\sup_{\bsx\in \mathrm{EI}(\bsell,2\bsa)}|\Delta(g,\bsx+J_{\bsell})|\Big)^2 \prod_{j=1}^s 4^{-\ell_j}\prod_{\ell'\in \kappa^+_j}4^{-\ell'-1}\\
&=2^{-s}\Bigl(\prod_{j=1}^s 4^{-\lambda \ell_j}\prod_{\ell'\in \kappa^+_j}4^{-\ell'-1}\Bigr)\sum_{\bsa\in A}\sup_{\bsx\in \mathrm{EI}(\bsell,2\bsa)}\mathrm{Vol}(J_\ell)\Bigl(\frac{\Delta(g,\bsx+J_{\bsell})}{\mathrm{Vol}(J_\ell)^\lambda}\Bigr)^2 
\end{align*}
For any $\bsx_{\bsa}\in \mathrm{EI}(\bsell,2\bsa)$, we know $\bsx_{\bsa}+J_{\bsell}\subseteq \mathrm{EI}(\bsell-\bsone_{1{:}s},\bsa)$ and the set of cubes $\{\bsx_{\bsa}+J_{\bsell}\mid\bsa\in A\}$ are mutually disjoint, so there exists a partition $\mathcal{P}$ of $[0,1)^s$ into subcubes that includes $\{\bsx_{\bsa}+J_{\bsell}\mid\bsa\in A\}$ as a subset. Hence
$$\sum_{\bsa\in A}\mathrm{Vol}(J_\ell)\Bigl(\frac{\Delta(g,\bsx_{\bsa}+J_{\bsell})}{\mathrm{Vol}(J_\ell)^\lambda}\Bigr)^2
\leq \sum_{J\in \mathcal{P}}\mathrm{Vol}(J)\Bigl|\frac{\Delta(g,J)}{\mathrm{Vol}(J)^\lambda}\Bigr|^2
\leq \Bigl(V^{(s)}_\lambda(g)\Bigr)^2. $$
Taking the supremum over $\bsx_{\bsa}\in \mathrm{EI}(\bsell,2\bsa)$ gives
\begin{align*}
\sum_{\kappa^-_1,\dots,\kappa^-_s}|\hat{f}(\bsk)|^2
\leq 2^{-s}\Bigl(V^{(s)}_\lambda(g)\Bigr)^2\Bigl(\prod_{j=1}^s 4^{-\lambda \ell_j}\prod_{\ell'\in \kappa^+_j}4^{-\ell'-1}\Bigr).
\end{align*}
When $\alpha=0$, $\kappa^+_j$ is empty and we have proven our conclusion. When $\alpha>0$, $\kappa^+_j$ can be any $\alpha$ distinct numbers larger than $\ell_j$. After summing over all possible choice of $\kappa^+_j$, we get
\begin{equation}\label{eqn:alphasum}
  \sum_{\kappa^+_j}\prod_{\ell'\in \kappa^+_j}4^{-\ell'-1}\leq  \frac{1}{\alpha!}\Bigl(\sum_{\ell'=\ell_j+1}^\infty 4^{-\ell'-1}\Bigr)^\alpha=\frac{4^{-\alpha \ell_j}12^{-\alpha}}{\alpha!}  
\end{equation}
where the inequality holds because every permutation of $\kappa^+_j$ is contained in the $\alpha$-fold product of all natural numbers larger than $\ell_j$. Summing over all possible choice of $\kappa^+_1,\dots,\kappa^+_s$, we finally get
\begin{align*}
\sum_{\kappa^+_1,\dots,\kappa^+_s}\sum_{\kappa^-_1,\dots,\kappa^-_s}|\hat{f}(\bsk)|^2
&\leq 2^{-s}\Bigl(V^{(s)}_\lambda(g)\Bigr)^2\Bigl(\prod_{j=1}^s 4^{-\lambda \ell_j }\frac{4^{-\alpha \ell_j}12^{-\alpha}}{\alpha!}\Bigr)\\
&\leq \Bigl(V^{(s)}_\lambda(g)\Bigr)^2 4^{-(\alpha+\lambda)\Vert\bsell\Vert_1}.
\end{align*}
The proof is complete once we notice the sum over all possible $\kappa^+_1,\dots,\kappa^+_s$ and $\kappa^-_1,\dots,\kappa^-_s$ is precisely the sum of $\bsk$ over $B_{\alpha,\bsell,s}$.
\end{proof}

\begin{corollary}\label{cor:fkboundVlambda}
 Suppose $N_{\alpha,\lambda}(f)<\infty$ for $\alpha\in \natu_0$ and $0<\lambda\le 1$. Then for $\bsell\in \natu_*^s$ and $B_{\alpha,\bsell,s}=\{\bsk\in \natu_*^s\mid \lceil\bsk\rceil_{\alpha+1}=\bsell\}$
$$\Bigl(\sum_{\bsk\in B_{\alpha,\bsell,s}}|\hat{f}(\bsk)|^2\Bigr)^{1/2}\leq N_{\alpha,\lambda}(f) 2^{-(\alpha+\lambda)\Vert\bsell\Vert_1}.$$
\end{corollary}
\begin{proof}
We have proven the result when $\bsell\in \natu^s$, so we assume $u=\{j\in 1{:}s\mid \ell_j>0\}$ is a proper subset of $1{:}s$. $u$ is also nonempty since $\bsell\neq \bszero$.

For $\bsk\in B_{\alpha,\bsell,s}$, we let $\bsalpha=(\alpha_1,\dots,\alpha_s)$ with $\alpha_j=\alpha$ if $j\in u$ and $\alpha_j=|\kappa_j|$ if $j\notin u$. For simplicity, we will write $\kappa_{u^c}$ as a shorthand for $\kappa_j,j\in u^c$. Equation~\eqref{eqn:exactfk} with the above choice of $\bsalpha$ gives
\begin{equation*}
    |\hat{f}(\bsk)|^2=\Bigl|\int_{[0,1)^{|u|}}g_{\kappa_{u^c}}(\bsx_{u})\prod_{j\in u}\walkappa{\kappa_j\setminus \lceil\kappa_j\rceil_{1{:}\alpha}}(x_j)W_{\lceil\kappa_j\rceil_{1{:}\alpha}}(x_j)\rd \bsx_{u}\Bigr|^2
\end{equation*}
where 
\begin{equation*}
    g_{\kappa_{u^c}}(\bsx_{u})=I_u\Big(f^{(\bsalpha)}\prod_{j\in u^c}W_{\kappa_j}\Big).
\end{equation*}

Let
\begin{equation*}
B'_{\kappa_{u^c}}=\{\bsk'\in B_{\alpha,\bsell,s}\mid \kappa'_j=\kappa_j \text{ for } j\in u^c\}.
\end{equation*}
Because every $\bsk'\in  B'_{\kappa_{u^c}}$ defines the same $g_{\kappa_{u^c}}$, by the proof of Theorem~\ref{thm:fkboundgealpha} with $g_{\kappa_{u^c}}$ in place of $g$
\begin{equation}\label{eqn:sumfk'}
    \sum_{\bsk'\in  B'_{\kappa_{u^c}}}|\hat{f}(\bsk')|^2\leq \Bigl(V^{(|u|)}_\lambda(g_{\kappa_{u^c}})\Bigr)^2 4^{-(\alpha+\lambda)\Vert\bsell\Vert_1}
\end{equation}
where we have used the fact $\Vert\bsell_u\Vert_1=\Vert\bsell\Vert_1$ because $\ell_j=0$ if $j\notin u$.

 Because $W_{\kappa_j}$ is continuous and nonnegative, we let $\rho_j=W_{\kappa_j}/\Vert W_{\kappa_j}\Vert_\infty$ and use the definition of $ V^{u}_\lambda$ as in equation~\eqref{eqn:fractionalvariation} to get
 \begin{align*}
    V^{(|u|)}_\lambda(g_{\kappa_{u^c}})&=\Big(\prod_{j\in u^c}\Vert W_{\kappa_j}\Vert_\infty\Big)V^{(|u|)}_\lambda\Big(I_u(f^{(\bsalpha)}(\bsx)\prod_{j\in u^c}\rho_j(x_j))\Big)\\
    &\leq \Big(\prod_{j\in u^c}\Vert W_{\kappa_j}\Vert_\infty\Big)V^{u}_\lambda(f^{(\bsalpha)}).
\end{align*}
    By equation~\eqref{eqn:Wmax}, if $\kappa_j\neq \emptyset$
$$\Vert W_{\kappa_j}\Vert_\infty= 2\prod_{\ell'\in \kappa_j}2^{-\ell'-1}\leq \prod_{\ell'\in \kappa_j}2^{-\ell'}$$
and $\Vert W_{\kappa_j}\Vert_\infty=1$ if  $\kappa_j= \emptyset$.

Summing equation~\eqref{eqn:sumfk'} over all possible choice of $\kappa_{u^c}$ such that $|\kappa_j|=\alpha_j, j\in u^c$, we get
\begin{align*}
    \sum_{\kappa_{u^c}: |\kappa_j|=\alpha_j}\sum_{\bsk'\in  B'_{\kappa_{u^c}}}|\hat{f}(\bsk')|^2
    &\leq 4^{-(\alpha+\lambda)\Vert\bsell\Vert_1}\sum_{\kappa_{u^c}: |\kappa_j|=\alpha_j}\Bigl(V^{(|u|)}_\lambda(g_{\kappa_{u^c}})\Bigr)^2 \\
    & \leq \Big(V^u_\lambda(f^{(\bsalpha)})\Bigr)^2 4^{-(\alpha+\lambda)\Vert\bsell\Vert_1}\sum_{\kappa_{u^c}: |\kappa_j|=\alpha_j}\prod_{j\in u^c:\kappa_j\neq \emptyset}\prod_{\ell'\in \kappa_j}4^{-\ell'}\\
    & \leq\Big(V^u_\lambda(f^{(\bsalpha)})\Bigr)^2 4^{-(\alpha+\lambda)\Vert\bsell\Vert_1} \prod_{j\in u^c:\kappa_j\neq \emptyset}\frac{1}{\alpha_j!}\Bigl(\sum_{\ell'=1}^\infty 4^{-\ell'}\Bigr)^{\alpha_j}\\
    & \leq\Big(V^u_\lambda(f^{(\bsalpha)})\Bigr)^2 4^{-(\alpha+\lambda)\Vert\bsell\Vert_1} .
\end{align*}
Finally, we sum over all choice of $\alpha_j\in \ints_{\le\alpha}, j\in u^c$ and get 
\begin{align}\label{eqn:corNalambda}
    \sum_{\kappa_{u^c}:|\kappa_j|\leq \alpha}\sum_{\bsk'\in  B'_{\kappa_{u^c}}}|\hat{f}(\bsk')|^2
    &\leq \sum_{\alpha_j\in \ints_{\le\alpha}, j\in u^c}\Big(V^u_\lambda(f^{(\bsalpha)})\Bigr)^2 4^{-(\alpha+\lambda)\Vert\bsell\Vert_1}\nonumber\\
    &\leq \Bigl(N_{\alpha,\lambda}(f)\Bigr)^2 4^{-(\alpha+\lambda)\Vert\bsell\Vert_1}.
\end{align}
The proof is complete because $B_{\alpha,\bsell,s}$ is the union of $B'_{\kappa_{u^c}}$ over all choice of $\kappa_{u^c}$ such that $|\kappa_j|\leq \alpha, j\in u^c$.
\end{proof}

Motivated by Lemma~\ref{lem:fkboundlealpha} and Theorem~\ref{thm:fkboundgealpha}, we define event $\mathcal{A}$ to be $\{Z(\bsk)=1$ for some $\bsk\in K\}$ with
$$K=\{\bsk\in \natu_*^s\mid \Vert\lceil\bsk\rceil_{\alpha+1}\Vert_1\leq m-(\alpha+1) s\log_2(m),\max\lceil\kappa_j\rceil\leq (\alpha+\lambda+1)m\}.$$

\begin{lemma}\label{lem:smallpVlambda}
    With $\mathcal{A}$ defined as above, we have $\Pr(\mathcal{A})=O(1/m)$.
\end{lemma}

\begin{proof} As in the proof of Lemma~\ref{lem:smallpVa}, we only need to show $|K|=O(2^m/m)$. To facilitate the proof, we define
\begin{equation}\label{eqn:KT}
   K_T=\{\bsk\in \natu_*^s\mid \Vert\lceil\bsk\rceil_{\alpha+1}\Vert_1\leq m-(\alpha+1) s\log_2(m),\max\lceil\kappa_j\rceil\leq T\}. 
\end{equation}
Consider $B_{\alpha,\bsell,s}=\{\bsk\in \natu_*^s\mid \lceil\bsk\rceil_{\alpha+1}=\bsell\}$ with $\Vert\bsell\Vert_1\leq m-(\alpha+1) s\log_2(m)$. $\lceil\kappa_j\rceil_{1{:}\alpha}$ can be any $\alpha$ distinct numbers larger than $\ell_j$, so there are less than $T^\alpha$ choices given $\max\lceil\kappa_j\rceil\leq T$. $\kappa_j\setminus\lceil\kappa_j\rceil_{1{:}\alpha+1}$ can be any subset of $1:\ell_j-1$ if $\ell_j\geq 2$ and is empty if $\ell_j=0$ or $1$, so there are at most $2^{\ell_j}$ choices. Hence
    \begin{align*}
        |K_T|&<\sum_{\substack{\bsell\in \natu_0^{ s}\\\Vert\bsell\Vert_1\leq m-(\alpha+1) s\log_2(m)}}\prod_{j=1}^s T^\alpha 2^{\ell_j}\\
        &=T^{\alpha s} \sum_{\substack{\bsell\in \natu_0^{ s}\\\Vert\bsell\Vert_1\leq m-(\alpha+1) s\log_2(m)}} 2^{\Vert\bsell\Vert_1}.
    \end{align*}
    Because 
    $$ |\{\bsell\in \natu_0^{s}\mid \Vert\bsell\Vert_1=N\}|={N+s-1\choose s-1},$$
    we can apply Lemma~\ref{lem:geometricsum} and get
    \begin{align}\label{eqn:Tbound}
        |K_T| & <T^{\alpha s} \sum_{N=0}^{\lfloor m-(\alpha+1) s \log_2(m)\rfloor}{N+s-1\choose s-1} 2^{N}\nonumber\\
        & \leq T^{\alpha s} {m-(\alpha+1) s \log_2(m)+s-1\choose s-1} 2^{m-(\alpha+1) s \log_2(m)+1}\nonumber\\
        &\leq T^{\alpha s} \frac{(m+s-1)^{s-1}}{(s-1)!} \frac{2^{m+1}}{m^{(\alpha+1) s }},
    \end{align}
which implies $|K|=O(2^m/m)$ once we let $T=(\alpha+\lambda+1)m$.
 
\end{proof}

\begin{theorem}\label{thm:rateVlambda}
     Suppose  $N_{\alpha,\lambda}(f)<\infty$ for $\alpha\in \natu_0$ and $0<\lambda\le 1$. When $r\geq m$, we have $\e((\hat{\mu}^{(r)}_{\infty}-\mu)^2)=O(n^{-2\alpha-2\lambda-1+\epsilon})$ for any $\epsilon>0$.
\end{theorem}

\begin{proof}
    As in the proof of Theorem~\ref{thm:rateVa}, it suffices to show
    $$\sum_{\bsk\in\natu_*^s\setminus K }\Pr(Z(\bsk)=1)|\hat{f}(\bsk)|^2=O(n^{-2\alpha-2\lambda-1+\epsilon}).$$
    Lemma~\ref{lem:PrZkbound} and the definition of $K$ imply
    \begin{align}
    &\sum_{\bsk\in\natu_*^s\setminus K }\Pr(Z(\bsk)=1)|\hat{f}(\bsk)|^2 \nonumber\\
    \leq & 2^{-m+R}\sum_{\bsk\in\natu_*^s\setminus K} |\hat{f}(\bsk)|^2\nonumber\\
    \leq &  2^{-m+R}\sum_{\substack{\bsk\in\natu^s_*\\\Vert\lceil\bsk\rceil_{\alpha+1}\Vert_1> m-(\alpha+1) s\log_2(m)}}|\hat{f}(\bsk)|^2\label{eqn:firstsum}\\ 
    + &  2^{-m+R}\sum_{\substack{\bsk\in\natu^s_*\\\Vert\lceil\bsk\rceil_{\alpha+1}\Vert_1\leq m-(\alpha+1) s\log_2(m)\\ \max\lceil\kappa_j\rceil> (\alpha+\lambda+1)m}}|\hat{f}(\bsk)|^2.\label{eqn:secondsum}
\end{align}
We apply Corollary~\ref{cor:fkboundVlambda} to the sum in line~\eqref{eqn:firstsum} and get
\begin{align*}
    &\sum_{\substack{\bsk\in\natu^s_*\\\Vert\lceil\bsk\rceil_{\alpha+1}\Vert_1> m-(\alpha+1) s\log_2(m)}}|\hat{f}(\bsk)|^2 \\
     = &\sum_{\substack{\bsell\in \natu_0^{s}\\\Vert\bsell\Vert_1> m-(\alpha+1) s \log_2(m)}} \sum_{\bsk\in B_{\alpha,\bsell,s}}|\hat{f}(\bsk)|^2\\
    \leq & \sum_{\substack{\bsell\in \natu_0^{s}\\\Vert\bsell\Vert_1> m-(\alpha+1) s \log_2(m)}} \Bigl(N_{\alpha,\lambda}(f)\Bigr)^2 4^{-(\alpha+\lambda)\Vert\bsell\Vert_1}
\end{align*}
By a similar argument used in equation~\eqref{eqn:epilsonrate}, we have 
$$\sum_{\substack{\bsell\in \natu_0^{s}\\\Vert\bsell\Vert_1> m-(\alpha+1) s \log_2(m)}} 4^{-(\alpha+\lambda)\Vert\bsell\Vert_1}=O(2^{-(2\alpha+2\lambda-\epsilon)m}).$$
Therefore when $N_{\alpha,\lambda}(f)<\infty$
$$\sum_{\substack{\bsk\in\natu^s_*\\\Vert\lceil\bsk\rceil_{\alpha+1}\Vert_1> m-(\alpha+1) s\log_2(m)}}|\hat{f}(\bsk)|^2=O(2^{-(2\alpha+2\lambda-\epsilon)m}).$$
Next we notice the sum in line~\eqref{eqn:secondsum} is empty if $\alpha=0$ because $\Vert\lceil\bsk\rceil_{1}\Vert_1\geq \max\lceil\kappa_j\rceil$. When $\alpha\geq 1$, we apply Lemma~\ref{lem:fkboundlealpha} and get
\begin{align*}
    \sum_{\substack{\bsk\in\natu^s_*\\\Vert\lceil\bsk\rceil_{\alpha+1}\Vert_1\leq m-(\alpha+1) s\log_2(m)\\ \max\lceil\kappa_j\rceil> (\alpha+\lambda+1)m}}|\hat{f}(\bsk)|^2 \leq C_f^2\sum_{\substack{\bsk\in\natu^s_*\\\Vert\lceil\bsk\rceil_{\alpha+1}\Vert_1\leq m-(\alpha+1) s\log_2(m)\\ \max\lceil\kappa_j\rceil> (\alpha+\lambda+1)m}} 4^{-\max \lceil \kappa_j\rceil_1}. 
\end{align*}
With $K_T$ defined as in equation~\eqref{eqn:KT}, we can apply the bound from equation~\eqref{eqn:Tbound} and get
    \begin{align*}       &\sum_{\substack{\bsk\in\natu^s_*\\
    \Vert\lceil\bsk\rceil_{\alpha+1}\Vert_1\leq m-(\alpha+1) s\log_2(m)\\ \max\lceil\kappa_j\rceil> (\alpha+\lambda+1)m}} 4^{-\max \lceil \kappa_j\rceil_1} \\
    = &\sum_{T=\lfloor(\alpha+\lambda+1)m\rfloor+1}^\infty |K_T\setminus K_{T-1}|4^{-T}\\
        \leq &\frac{(m+s-1)^{s-1}}{(s-1)!} \frac{2^{m+1}}{m^{(\alpha+1) s }} \sum_{T=\lfloor(\alpha+\lambda+1)m\rfloor+1}^\infty T^{\alpha s} 4^{-T}.
    \end{align*}
    Because $T^{\alpha s}=O(4^{\delta T})$ for any $\delta>0$,
    $$ \sum_{T=\lfloor(\alpha+\lambda+1)m\rfloor+1}^\infty T^{\alpha s} 4^{-T}=O(\sum_{T=\lfloor(\alpha+\lambda+1)m\rfloor+1}^\infty 4^{-(1-\delta)T})=O(2^{-2(1-\delta)(\alpha+\lambda+1)m}).$$
    Hence by choosing a small enough $\delta$,
    \begin{align*}
         \sum_{\substack{\bsk\in\natu^s_*\\\Vert\lceil\bsk\rceil_{\alpha+1}\Vert_1\leq m-(\alpha+1) s\log_2(m)\\ \max\lceil\kappa_j\rceil> (\alpha+\lambda+1)m}}|\hat{f}(\bsk)|^2 =O(2^m2^{-2(1-\delta)(\alpha+\lambda+1)m})=o(2^{-(2\alpha+2\lambda)m})
    \end{align*}
    Finally putting together the bounds for sums in line~\eqref{eqn:firstsum} and \eqref{eqn:secondsum}, we conclude
    \begin{align*}
       \sum_{\bsk\in\natu_*^s\setminus K }\Pr(Z(\bsk)=1)|\hat{f}(\bsk)|^2
        &\leq  2^{-m+R}O(2^{-(2\alpha+2\lambda-\epsilon)m})=O(n^{-2\alpha-2\lambda-1+\epsilon})
    \end{align*}
    since $n=2^m$.
\end{proof}

\begin{remark}\label{rmk:rateVlambda}
  Theorem~\ref{thm:rateVlambda} can be generalized to integrands with ``weak derivatives" $f^{(\bsalpha)}\in L^1([0,1]^s)$ that fulfill equation~\eqref{eqn:exactfk}. In particular, $f^{(\bsalpha)}\in C([0,1]^s)$ is assumed in Lemma~\ref{lem:exactfk} only to validate integration by part and derive equation~\eqref{eqn:exactfk}, and Lemma~\ref{lem:fkboundlealpha} and Theorem~\ref{thm:fkboundgealpha} are both direct consequences of equation~\eqref{eqn:exactfk}. A noticeable case where such generalization can always be done is when $\bsalpha=\bszero$, for which equation~\eqref{eqn:exactfk} holds with $f^{(\bsalpha)}=f$ because this is the definition of $\hat{f}(\bsk)$. Subsection~\ref{subsec:discont} gives a one-dimensional example where equation~\eqref{eqn:exactfk} holds despite a point of discontinuity in the derivative, and we see there the $O(n^{-\alpha-\lambda-1/2+\epsilon})$ RMSE convergence rate holds despite $f^{(\alpha)}\notin C([0,1])$. 
\end{remark}

\subsection{When $f$ is in the Sobolev space of order $\alpha+\lambda$}\label{subsec:Hp}

Sobolev spaces, especially reproducing kernel Hilbert spaces, have been used extensively in the study of QMC \cite{dick:pill:2010}. In this subsection, we will show the RMSE of $\hat{\mu}^{(r)}_\infty$ has a $O(n^{-\alpha-\lambda-1/2+\epsilon})$ convergence rate when the dominating mixed derivatives of $f$ up to order $\alpha$ belong to the fractional Sobolev space $\mathcal{H}_{\lambda,s,p}$ with $p\ge 2$ as defined below.

Following \cite{gnewuch2024qmcintegrationbasedarbitrary}, we say $f\in \mathcal{H}_{\lambda,s,p}$ with $0<\lambda\leq 1$ and $p>\lambda^{-1}$ if there exists a set of ``fractional derivatives" $D^\lambda_uf,u\subseteq 1{:}s$ such that $D^\lambda_uf \in L^p([0,1]^{|u|})$ and 
\begin{equation}\label{eqn:Df}
   f(\bsx)=\sum_{u\subseteq 1{:}s} \Gamma(\lambda)^{-|u|}\int_{[0,1]^{|u|}} D^\lambda_uf (\bst_u)\prod_{j\in u} (x_j-t_j)_+^{\lambda-1}\rd \bst_u 
\end{equation}
where $\Gamma(\lambda)$ is the Gamma function and $(x-t)_+=\max(x-t,0)$. When $u=\emptyset$, $D^\lambda_\emptyset f $ is simply a real number with $L^p$-norm equal to its absolute value, and the integral over $[0,1]^{|\emptyset|}$ evaluates to $D^\lambda_\emptyset f$. Lemma 7 of \cite{gnewuch2024qmcintegrationbasedarbitrary} shows $D^\lambda_uf$ is unique in the $L_p$ sense, so we will refer to $D^\lambda_uf $ as the $\lambda$-fractional derivative of $f$ in variables $u$. When $p\geq 2$, we further define the norm 
$$\Vert f\Vert_{\lambda,s}=\Bigl(\sum_{u\subseteq 1{:}s} \Vert D^\lambda_uf \Vert^2_{L^2([0,1]^{|u|})}\Bigr)^{1/2}.$$
Notice when $\lambda=1$, this norm is equal to the usual Sobolev norm of order $1$ anchored at $\bszero$ \cite{JOSEF:2008}.

Lemma 8 of \cite{gnewuch2024qmcintegrationbasedarbitrary} proves $f\in \mathcal{H}_{\lambda,s,p}$ is Hölder continuous with Hölder exponent $\lambda-p^{-1}$, so in particular $f\in \mathcal{H}_{\lambda,s,p}$ implies $f\in C([0,1]^s)$. A slightly modified proof shows the fractional variation $V^{u}_{\lambda-p^{-1}}(f)$ is also finite if $f\in \mathcal{H}_{\lambda,s,p}$. Therefore, we can apply Theorem~\ref{thm:rateVlambda} and prove the $O(n^{-\alpha-\lambda-1/2+\epsilon})$ RMSE convergence rate if $f^{(\bsalpha)}\in \mathcal{H}_{\lambda,s,p}$ with $p=\infty$ for all $\bsalpha\in \ints_{\le\alpha}^s$. We are going to prove the same rate holds under the weaker assumption $p\geq 2$. 

We will use the norm 
$$\Vert f\Vert_{\alpha,\lambda}=\Bigl(\sum_{\bsalpha\in \ints_{\le\alpha}^{s}}\Vert f^{(\bsalpha)}\Vert^2_{\lambda,s}\Bigr)^{1/2} .$$

First, we present a lemma showing $\mathcal{H}_{\lambda,s,p}$ can be embedded into $\mathcal{H}_{\lambda',s,p'}$ with $\lambda'<\lambda$ but $p'>p$. The proof technique is useful when we later prove Theorem~\ref{thm:fkboundHp}.

\begin{lemma}\label{lem:increasep}
    For $1\leq p <p' <\infty $ and $\lambda \in  (1/p-1/p', 1]$, $f\in \mathcal{H}_{\lambda,s,p}$ implies $f\in \mathcal{H}_{\lambda',s,p'}$ with any $\lambda'\in (0, \lambda-1/p+1/p')$.
\end{lemma}

\begin{proof}
    By the beta integral identity
    \begin{align*}
        \int_0^1 (x-y)^{\lambda'-1}_+  (y-t)_+^{\lambda-\lambda'-1}\rd y
        &=(x-t)^{\lambda-1}_+\int_{0}^1 (1-y)^{\lambda'-1} y^{\lambda-\lambda'-1}\rd y\\
        &=(x-t)^{\lambda-1}_+\frac{\Gamma(\lambda')\Gamma(\lambda-\lambda')}{\Gamma(\lambda)}.
    \end{align*}
    Hence for any $u\subseteq 1{:}s$
    \begin{align*}
       &\Gamma(\lambda)^{-|u|}\int_{[0,1]^{|u|}} D^\lambda_u f (\bst_u)\prod_{j\in u} (x_j-t_j)_+^{\lambda-1}\rd \bst_u\\
    =&\Gamma(\lambda)^{-|u|}\int_{[0,1]^{|u|}} D^\lambda_u f (\bst_u)\Big(\prod_{j\in u} \frac{\Gamma(\lambda)}{\Gamma(\lambda')\Gamma(\lambda-\lambda')} \int_0^1 (x_j-y_j)^{\lambda'-1}_+  (y_j-t_j)_+^{\lambda-\lambda'-1}\rd y_j\Big)\rd \bst_u \\
    =& \Gamma(\lambda')^{-|u|}\int_{[0,1]^{|u|}}\int_{[0,1]^{|u|}}  D^\lambda_u f (\bst_u)\Big(\prod_{j\in u} \frac{1}{\Gamma(\lambda-\lambda')} (x_j-y_j)^{\lambda'-1}_+  (y_j-t_j)_+^{\lambda-\lambda'-1}\Big)\rd \bsy_u\rd \bst_u\\
    =& \Gamma(\lambda')^{-|u|}\int_{[0,1]^{|u|}} D_u^{\lambda'}f (\bsy_u)\prod_{j\in u} (x_j-y_j)^{\lambda'-1}_+ \rd \bsy_u
    \end{align*}
    where 
    \begin{equation}\label{eqn:changelambda}
        D_u^{\lambda'}f (\bsy_u)=\Gamma(\lambda-\lambda')^{-|u|} \int_{[0,1]^{|u|}}D^\lambda_u f (\bst_u)\prod_{j\in u} (y_j-t_j)_+^{\lambda-\lambda'-1}\rd \bst_u.
    \end{equation}
    To show $D_u^{\lambda'}f$ is the $\lambda'$-fractional derivative of $f$ in variables $u$, it remains to show $D_u^{\lambda'}f\in L^{p'}([0,1]^{|u|})$. First note that $\Gamma(\lambda-\lambda')^{|u|}D_u^{\lambda'}f(\bsy_u)$ is the convolution of $D^\lambda_u f(\bst_u) \bsone\{\bst_u\in [0,1]^{|u|}\}$ with $(\prod_{j\in u} t_j)^{\lambda-\lambda'-1}\bsone\{\bst_u\in [0,1]^{|u|}\}$ over $\mathbb{R}^{|u|}$ when $\bsy_u$ is restricted to $[0,1]^{|u|}$. Extending the domain of $D_u^{\lambda'}f$ to $\mathbb{R}^{|u|}$ using the above convolution interpretation, we can apply Young's convolution inequality \cite{bahouri2011fourier} and get for $q=(1-1/p+1/p')^{-1}$
    \begin{align*}
        &\Gamma(\lambda-\lambda')^{|u|}\Vert D_u^{\lambda'}f \Vert_{L^{p'}([0,1]^{|u|})} \\
         \leq & \Gamma(\lambda-\lambda')^{|u|}\Vert D_u^{\lambda'}f \Vert_{L^{p'}(\mathbb{R}^{|u|})}\\
        \leq  &\Vert D^\lambda_u f(\bst_u) \bsone\{\bst_u\in [0,1]^{|u|}\}\Vert_{L^p(\mathbb{R}^{|u|})} \Vert (\prod_{j\in u} t_j)^{\lambda-\lambda'-1}\bsone\{\bst_u\in [0,1]^{|u|}\} \Vert_{L^q(\mathbb{R}^{|u|})}\\
        = &\Vert D^\lambda_u f\Vert_{L^p([0,1]^{|u|})} \Big(\prod_{j\in u}\int_0^1 t^{q(\lambda-\lambda'-1)}_j \rd t_j\Big)^{1/q}.
    \end{align*}
    Because $\lambda'<\lambda-1/p+1/p'$, $q(\lambda-\lambda'-1)>-q(1-1/p+1/p')=-1$. Therefore, $t^{q(\lambda-\lambda'-1)}_j$ is integrable over $[0,1]$ and $\Vert D_u^{\lambda'}f \Vert_{L^q([0,1]^{|u|})} <\infty$. Since this applies for all $u\subseteq 1{:}s$, we conclude  $f\in \mathcal{H}_{\lambda',s,p'}$.
\end{proof}

\begin{corollary}\label{cor:plessthan2}
Suppose $f^{(\bsalpha)}\in \mathcal{H}_{\lambda,s,p}$ for all $\bsalpha\in \ints_{\le\alpha}^{s}$ with $\alpha\in \natu_0$, $p\in [1,2)$ and $\lambda\in (1/p-1/2, 1]$. Then $\Vert f\Vert_{\alpha,\lambda'}<\infty$ with any $\lambda'\in (0,\lambda-1/p+1/2)$.
\end{corollary}

\begin{proof}
    Applying Lemma~\ref{lem:increasep} with $p'=2$, we see $f^{(\bsalpha)}\in \mathcal{H}_{\lambda,s,p}$ implies $f^{(\bsalpha)}\in \mathcal{H}_{\lambda',s,2}$ with any $\lambda'\in (0,\lambda-1/p+1/2)$. Hence $\Vert f^{(\bsalpha)}\Vert_{\lambda',s}<\infty$ for all $\bsalpha\in \ints_{\le\alpha}^{s}$ and  $\Vert f\Vert_{\alpha,\lambda'}<\infty$.
\end{proof}

The above corollary shows we can focus our analysis on $p\geq 2$ case. The next theorem is the counterpart of Theorem~\ref{thm:fkboundgealpha}.

\begin{theorem}\label{thm:fkboundHp}
Let $\alpha\in \natu_0$ and $0<\lambda\le 1$. Suppose $f^{(\alpha,\dots,\alpha)}\in \mathcal{H}_{\lambda,s,p}$ for $p\geq 2$. 
 Then for $\bsell\in \natu^s$ and  $B_{\alpha,\bsell,s}=\{\bsk\in \natu_*^s\mid \lceil\bsk\rceil_{\alpha+1}=\bsell\}$
 \begin{equation}\label{eqn:fkboundHp}
     \Bigl(\sum_{\bsk\in B_{\alpha,\bsell,s}}|\hat{f}(\bsk)|^2\Bigr)^{1/2}\leq C^s_{\lambda}\Vert D^\lambda_{1{:}s} f^{(\alpha,\dots,\alpha)}\Vert_{L^2([0,1]^s)}2^{-(\alpha+\lambda)\Vert\bsell\Vert_1}
 \end{equation}
where $C_{\lambda}$ is a constant defined in equation~\eqref{eqn:Clambda}.
\end{theorem}

\begin{proof}
    We will use the notation from the proof of Theorem~\ref{thm:fkboundgealpha}. Again let $g(\bsx)=f^{(\alpha,\dots,\alpha)}(\bsx)$. First we note that Lemma~\ref{lem:exactfk} still applies here because  $g\in \mathcal{H}_{\lambda,s,p}$ implies $g\in C([0,1]^s)$. By equation~\eqref{eqn:fktodeltaf} and \eqref{eqn:fksquaresum},
    \begin{equation}\label{eqn:kappaminussum}
        \sum_{\kappa^-_1,\dots,\kappa^-_s}|\hat{f}(\bsk)|^2=|A|\sum_{\bsa\in A}\Big|\int_{\mathrm{EI}(\bsell,2\bsa)} \Delta(g,\bsx+J_{\bsell})\prod_{j=1}^s W_{\kappa^+_j}(x_j)\rd \bsx\Big|^2.
    \end{equation}
Using equation~\eqref{eqn:Wmax} to bound $W_{\kappa^+_j}$ and applying Hölder's inequality, we get
\begin{align}\label{eqn:fkboundHp}
    &\sum_{\kappa^-_1,\dots,\kappa^-_s}|\hat{f}(\bsk)|^2 \nonumber\\
    \leq & 4^s\Bigl(\prod_{j=1}^s\prod_{\ell'\in\kappa^+_j} 4^{-\ell'-1} \Bigr)|A|\sum_{\bsa\in A}\Big(\int_{\mathrm{EI}(\bsell,2\bsa)} |\Delta(g,\bsx+J_{\bsell})|\rd \bsx\Big)^2 \nonumber\\
    \leq & 4^s\Bigl(\prod_{j=1}^s\prod_{\ell'\in\kappa^+_j} 4^{-\ell'-1} \Bigr)|A|\mathrm{Vol}(\mathrm{EI}(\bsell,2\bsa))\sum_{\bsa\in A}\int_{\mathrm{EI}(\bsell,2\bsa)} |\Delta(g,\bsx+J_{\bsell})|^2\rd \bsx \nonumber \\
    \leq & 2^s \Bigl(\prod_{j=1}^s\prod_{\ell'\in\kappa^+_j} 4^{-\ell'-1} \Bigr)  \int_{[0,1]^s} |\Delta(g,\bsx+J_{\bsell})|^2 \prod_{j=1}^s \bsone\{x_j\in [0,1-2^{-\ell_j})\} \rd \bsx,
\end{align}
where we have used $\mathrm{EI}(\bsell,2\bsa)$ are disjoint for $\bsa\in A$ and $|A|\mathrm{Vol}(\mathrm{EI}(\bsell,2\bsa))=2^{-s}|A|\mathrm{Vol}(\mathrm{EI}(\bsell-\bsone_{1{:}s},\bsa))=2^{-s}$ in the last inequality.
    
    By equation~\eqref{eqn:deltaf} and \eqref{eqn:Df}, when $x_j\in [0,1-2^{-\ell_j})$ for all $j\in 1{:}s$,
$$\Delta(g,\bsx+J_{\bsell})=(\lambda\Gamma(\lambda))^{-s}\int_{[0,1]^{s}} D^\lambda_{1{:}s}g (\bst)\prod_{j=1}^s \Delta_\lambda(x_j,t_j,\ell_j)\rd \bst$$
with 
$$\Delta_\lambda(x,t,\ell)=\lambda(x-t+2^{-\ell})_+^{\lambda-1}-\lambda (x-t)_+^{\lambda-1}.$$
Define 
$$F_{\lambda,\bsell}(\bst)=\prod_{j=1}^s \Big(\lambda(t_j+2^{-\ell_j})_+^{\lambda-1}-\lambda(t_j)_+^{\lambda-1}\Big).$$
Then $(\lambda\Gamma(\lambda))^{s}\Delta(g,\bsx+J_{\bsell})$ can be viewed as the convolution of $D^\lambda_{1{:}s}g (\bst) \bsone\{\bst\in [0,1]^s\} $ with $F_{\lambda,\bsell}(\bst)$ over $\mathbb{R}^{s}$. Like in the proof of Lemma~\ref{lem:increasep}, we extend the domain of $\Delta(g,\bsx+J_{\bsell})$ to $\bsx\in \mathbb{R}^{s}$ using this convolution interpretation and apply Young's convolution inequality to get
\begin{align}\label{eqn:Young}
   & (\lambda\Gamma(\lambda))^{2s}\int_{[0,1]^s} |\Delta(g,\bsx+J_{\bsell})|^2 \prod_{j=1}^s \bsone\{x_j\in [0,1-2^{-\ell_j})\} \rd \bsx \nonumber\\
   \leq & \Vert (\lambda\Gamma(\lambda))^{s} \Delta(g,\bsx+J_{\bsell}) \Vert^2_{L^2(\mathbb{R}^{s})} \nonumber\\
   \leq & \Vert D^\lambda_{1{:}s}g (\bst) \bsone\{\bst\in [0,1]^s\} \Vert^2_{L^2(\mathbb{R}^{s})} \Vert F_{\lambda,\bsell}(\bst)\Vert^2_{L^1(\mathbb{R}^{s})}\nonumber\\
   = &\Vert D^\lambda_{1{:}s}g (\bst)\} \Vert^2_{L^2([0,1]^s)} \Big(\int_{\mathbb{R}^{s}} |F_{\lambda,\bsell}(\bst)|\rd \bst\Big)^2.
\end{align}
By the definition of $F_{\lambda,\bsell}(\bst)$, we can further write
\begin{align}\label{eqn:FL1norm}
    \int_{\mathbb{R}^{s}} |F_{\lambda,\bsell}(\bst)|\rd \bst = &\prod_{j=1}^s \Big( \int_{-2^{-\ell_j}}^0 \lambda(t_j+2^{-\ell_j})^{\lambda-1} \rd t_j +\int_{0}^\infty \lambda t_j^{\lambda-1}- \lambda(t_j+2^{-\ell_j})^{\lambda-1} \rd t_j \Big) \nonumber \\
    =& \prod_{j=1}^s \Big( 2^{-\lambda\ell_j} + \lim_{t_j\to\infty} t^\lambda_j - (t_j+2^{-\ell_j})^\lambda + 2^{-\lambda\ell_j}  \Big) \nonumber \\
    = & \begin{cases}
        2^{s-\lambda \Vert\bsell\Vert_1}& \text{ if } \lambda\in (0,1),\\
         2^{-\Vert\bsell\Vert_1} &\text{ if } \lambda=1.
    \end{cases}
\end{align}
Putting equation~\eqref{eqn:Young} and \eqref{eqn:FL1norm} into equation~\eqref{eqn:fkboundHp}, we get
$$ \sum_{\kappa^-_1,\dots,\kappa^-_s}|\hat{f}(\bsk)|^2\leq C^{2s}_{\lambda}\Vert D^\lambda_{1{:}s}g\Vert^2_{L^2([0,1]^s)}4^{-\lambda \Vert\bsell\Vert_1}\Bigl(\prod_{j=1}^s\prod_{\ell'\in\kappa^+_j} 4^{-\ell'-1} \Bigr)$$
with 
\begin{equation}\label{eqn:Clambda}
    C_{\lambda}=\begin{cases}2\sqrt{2} \Gamma(\lambda+1)^{-1}
     &  \text{ if } 0<\lambda<1,\\
    \sqrt{2} &\text{ if } \lambda=1.
    \end{cases}
\end{equation}
Notice we have used $\Gamma(\lambda+1)=\lambda \Gamma(\lambda)$ to simplify the expression.

Finally by equation~\eqref{eqn:alphasum}, 
\begin{align*}
\sum_{\kappa^+_1,\dots,\kappa^+_s}\sum_{\kappa^-_1,\dots,\kappa^-_s}|\hat{f}(\bsk)|^2
&\leq C^{2s}_{\lambda}\Vert D^\lambda_{1{:}s} g\Vert^2_{L^2([0,1]^s)}4^{-\lambda \Vert\bsell\Vert_1}\Bigl(\prod_{j=1}^s\frac{4^{-\alpha \ell_j}12^{-\alpha}}{\alpha!}\Big)\\
&\leq C^{2s}_{\lambda}\Vert D^\lambda_{1{:}s} g\Vert^2_{L^2([0,1]^s)}4^{-(\alpha+\lambda)\Vert\bsell\Vert_1}
\end{align*}
which completes the proof since the left sum is equivalent to the sum of $\bsk$ over $B_{\alpha,\bsell,s}$.
\end{proof}

\begin{corollary}\label{cor:fkboundHp}
    Suppose $\Vert f\Vert_{\alpha,\lambda}<\infty$ for $\alpha\in \natu_0$ and $0<\lambda\le 1$. Then for $\bsell\in \natu_*^s$ and $B_{\alpha,\bsell,s}=\{\bsk\in \natu_*^s\mid \lceil\bsk\rceil_{\alpha+1}=\bsell\}$
$$\Bigl(\sum_{\bsk\in B_{\alpha,\bsell,s}}|\hat{f}(\bsk)|^2\Bigr)^{1/2}\leq C^s_{\lambda}\Vert f\Vert_{\alpha,\lambda} 2^{-(\alpha+\lambda)\Vert\bsell\Vert_1}$$
where $C_{\lambda}$ is the constant defined in Theorem~\ref{thm:fkboundHp}.
\end{corollary}

\begin{proof}

We again let $u=\{j\in 1{:}s\mid \ell_j>0\}$ and $\bsalpha=(\alpha_1,\dots,\alpha_s)$ with $\alpha_j=\alpha$ if $j\in u$. By the proof of Corollary~\ref{cor:fkboundVlambda} with equation~\eqref{eqn:sumfk'} replaced by
$$ \sum_{\bsk'\in  B'_{\kappa_{u^c}}}|\hat{f}(\bsk')|^2\leq C^{2|u|}_{\lambda} \Vert D^\lambda_{u} g_{\kappa_{u^c}}\Vert^2_{L^2([0,1]^{|u|})} 4^{-(\alpha+\lambda)\Vert\bsell\Vert_1},$$
equation~\eqref{eqn:corNalambda} becomes
\begin{equation}\label{eqn:corfalambda}
    \sum_{\bsk\in B_{\alpha,\bsell,s}}|\hat{f}(\bsk)|^2\leq C^{2|u|}_{\lambda} 4^{-(\alpha+\lambda)\Vert\bsell\Vert_1}\sum_{\alpha_j\in \ints_{\le\alpha}, j\in u^c} \Vert f^{(\bsalpha)} \Vert^2_{\lambda,u} 
\end{equation}
where 
$$\Vert g \Vert_{\lambda,u}=\sup_{\rho_{u^c}\in D_{u^c}} \Vert D^\lambda_u I_u(g\rho_{u^c})\Vert_{L^2([0,1]^{|u|})}$$
for $g\in \mathcal{H}_{\lambda,s,p}$ with $I_u$ and $D_{u^c}$ defined in equation~\eqref{eqn:Iu} and \eqref{eqn:Du}. By equation~\eqref{eqn:Df},
\begin{align*}
    &I_u(g\rho_{u^c})(\bsx_u)\\
    =&\int_{[0,1]^{s-|u|}}\Big(\sum_{v\subseteq 1{:}s} \Gamma(\lambda)^{-|v|}\int_{[0,1]^{|v|}} D^\lambda_v g (\bst_v)\prod_{j\in v} (x_j-t_j)_+^{\lambda-1}\rd \bst_v\Big) \prod_{j\in u^c}\rho_j(x_j)\rd \bsx_{u^c} \\
    =&\sum_{w\subseteq u} \Gamma(\lambda)^{-|w|}\int_{[0,1]^{|w|}}\Big(   \sum_{w'\subseteq u^c}\Gamma(\lambda)^{-|w'|}\int_{[0,1]^{|w'|}}D^\lambda_{w\cup w'} g (\bst_{w\cup w'}) \\
    &\int_{[0,1]^{s-|u|}}\prod_{j\in w'} (x_j-t_j)_+^{\lambda-1}\prod_{j\in u^c}\rho_j(x_j)\rd \bsx_{u^c}\rd \bst_{w'} \Big) \prod_{j\in w} (x_j-t_j)_+^{\lambda-1} \rd \bst_w.
\end{align*}
Hence by the definition of fractional derivative
\begin{align*}
    D^\lambda_uI_u(g\rho_{u^c})(\bsx_u)= &\sum_{w'\subseteq u^c}\Gamma(\lambda)^{-|w'|}\int_{[0,1]^{|w'|}}D^\lambda_{u\cup w'} g (\bst_{ w'};\bsx_u) \\
    &\int_{[0,1]^{s-|u|}}\prod_{j\in w'} (x_j-t_j)_+^{\lambda-1}\prod_{j\in u^c}\rho_j(x_j)\rd \bsx_{u^c}\rd \bst_{w'}.
\end{align*}
Using $\Vert\rho_j\Vert_\infty\leq 1$, we get the bound 
\begin{align*}
    &|D^\lambda_uI_u(g\rho_{u^c})(\bsx_u)|\\
    \leq &\sum_{w'\subseteq u^c}\Gamma(\lambda)^{-|w'|}\int_{[0,1]^{|w'|}}|D^\lambda_{u\cup w'} g (\bst_{ w'};\bsx_u)| \int_{[0,1]^{s-|u|}}\prod_{j\in w'} (x_j-t_j)_+^{\lambda-1}\rd \bsx_{u^c}\rd \bst_{w'}\\
    =&\sum_{w'\subseteq u^c}(\lambda\Gamma(\lambda))^{-|w'|}\int_{[0,1]^{|w'|}}|D^\lambda_{u\cup w'} g (\bst_{ w'};\bsx_u)|\prod_{j\in w'} (1-t_j)^\lambda\rd \bst_{w'}\\
    \leq & \sum_{w'\subseteq u^c}(\lambda\Gamma(\lambda))^{-|w'|}\Big(\int_{[0,1]^{|w'|}}|D^\lambda_{u\cup w'} g (\bst_{ w'};\bsx_u)|^2\rd\bst_{w'}\Big)^{1/2}\Big(\int_{[0,1]^{|w'|}}\prod_{j\in w'} (1-t_j)^{2\lambda}\rd\bst_{w'}\Big)^{1/2}\\
    = & \sum_{w'\subseteq u^c}\Gamma(\lambda+1)^{-|w'|}(2\lambda+1)^{-|w'|/2}\Big(\int_{[0,1]^{|w'|}}|D^\lambda_{u\cup w'} g (\bst_{ w'};\bsx_u)|^2\rd\bst_{w'}\Big)^{1/2}
\end{align*}
where we have used $\Gamma(\lambda+1)=\lambda\Gamma(\lambda)$ again. Applying the Cauchy–Schwarz inequality, we get
\begin{align*}
    &|D^\lambda_uI_u(g\rho_{u^c})(\bsx_u)|^2\\
    \leq  & \Big(\sum_{w'\subseteq u^c}\Gamma(\lambda+1)^{-2|w'|}(2\lambda+1)^{-|w'|}\Big)\Big(\sum_{w'\subseteq u^c}\int_{[0,1]^{|w'|}}|D^\lambda_{u\cup w'} g (\bst_{ w'};\bsx_u)|^2\rd\bst_{w'}\Big)\\
    =& (1+\Gamma(\lambda+1)^{-2}(2\lambda+1)^{-1})^{|u^c|}\sum_{w'\subseteq u^c}\int_{[0,1]^{|w'|}}|D^\lambda_{u\cup w'} g (\bst_{ w'};\bsx_u)|^2\rd\bst_{w'}.
\end{align*}
Therefore, after taking the supremum over $\rho_{u^c}\in D_{u^c}$
\begin{align*}
&\Vert g \Vert^2_{\lambda,u}\\
=&\sup_{\rho_{u^c}\in D_{u^c}} \Vert D^\lambda_u I_u(g\rho_{u^c})\Vert^2_{L^2([0,1]^{|u|})}\\
   \leq&(1+\Gamma(\lambda+1)^{-2}(2\lambda+1)^{-1})^{|u^c|}\sum_{w'\subseteq u^c}\int_{[0,1]^{|u|}}\int_{[0,1]^{|w'|}}|D^\lambda_{u\cup w'} g (\bst_{ w'};\bsx_u)|^2 \rd \bst_{w'} \rd \bsx_u \\
   =& (1+\Gamma(\lambda+1)^{-2}(2\lambda+1)^{-1})^{|u^c|}\sum_{w'\subseteq u^c}\Vert D^\lambda_{u\cup w'}g \Vert^2_{L^2[0,1]^{|u\cup w'|}}\\
   \leq & (1+\Gamma(\lambda+1)^{-2}(2\lambda+1)^{-1})^{|u^c|} \Vert g\Vert^2_{\lambda,s}.
\end{align*}
Putting the above bound into equation~\eqref{eqn:corfalambda}, we finally get
\begin{align*}
    &\sum_{\bsk\in B_{\alpha,\bsell,s}}|\hat{f}(\bsk)|^2\\
    \leq &C^{2|u|}_{\lambda} 4^{-(\alpha+\lambda)\Vert\bsell\Vert_1}(1+\Gamma(\lambda+1)^{-2}(2\lambda+1)^{-1})^{|u^c|}\sum_{\alpha_j\in \ints_{\le\alpha}, j\in u^c} \Vert f^{(\bsalpha)} \Vert^2_{\lambda,s}\\
    \leq & C^{2s}_{\lambda} 4^{-(\alpha+\lambda)\Vert\bsell\Vert_1} \Vert f\Vert^2_{\alpha,\lambda}
\end{align*}
where we have used the fact that $\Gamma(\lambda+1)\leq 1$ for $0<\lambda\leq 1$ and hence
\begin{align*}
   C^2_\lambda&\geq 2\Gamma(\lambda+1)^{-2} > 1+\Gamma(\lambda+1)^{-2}(2\lambda+1)^{-1}.\qedhere
\end{align*}
\end{proof}

\begin{theorem}\label{thm:rateHp}
    Suppose $ \Vert f\Vert_{\alpha,\lambda}<\infty$ for $\alpha\in \natu_0$ and $0<\lambda\le 1$. When $r\geq m$, we have $\e((\hat{\mu}^{(r)}_{\infty}-\mu)^2)=O(n^{-2\alpha-2\lambda-1+\epsilon})$ for any $\epsilon>0$.
\end{theorem}

\begin{proof}
    The proof is the same as that of Theorem~\ref{thm:rateVlambda} except we use Corollary~\ref{cor:fkboundHp} in place of Corollary~\ref{cor:fkboundVlambda}. Notice that Lemma~\ref{lem:fkboundlealpha} still applies when $\alpha\geq 1$ because $f^{(\bsalpha)}\in \mathcal{H}_{\lambda,s,p}$ for $\bsalpha\in \ints^s_{\le \alpha}$ implies $f\in C^{(\alpha,\dots,\alpha)}([0,1]^s)$.
\end{proof}

\begin{corollary}\label{cor:rateHp}
    Suppose $f^{(\bsalpha)}\in \mathcal{H}_{\lambda,s,p}$ for all $\bsalpha\in \ints_{\le\alpha}^{s}$ with $\alpha\in \natu_0$, $p\in [1,2)$ and $\lambda\in (1/p-1/2, 1]$. When  $r\geq m$, we have $\e((\hat{\mu}^{(r)}_{\infty}-\mu)^2)=O(n^{-2\alpha-2\lambda-2+2p^{-1}+\epsilon})$ for any $\epsilon>0$.
\end{corollary}
\begin{proof}
    By Corollary~\ref{cor:plessthan2}, $ \Vert f\Vert_{\alpha,\lambda'}<\infty$ with $\lambda'=\lambda-1/p+1/2-\epsilon$ for an arbitrarily small $\epsilon>0$. The result immediately follows by applying Theorem~\ref{thm:rateHp}.
\end{proof}

\begin{remark}
    Our results here can be used to analyze $f$ satisfying Owen's boundary growth condition \cite{haltavoid}, which assumes for all $\bsalpha\in \{0,1\}^s$, $f^{(\bsalpha)}$ is continuous over the open box $(0,1)^s$ and 
    \begin{equation}\label{eqn:Owengrowth}
        |f^{(\bsalpha)}(\bsx)|\leq B \prod_{j=1}^s\phi(x_j)^{-A^*-\alpha_j} 
    \end{equation}
    with $\phi(x)=\min(x,1-x)$, $A^*\in (-1/2,1/2)$ and $B>0$. To simplify our notation, we will write $\partial^u f$ to indicate $f^{(\bsalpha^u)}$ with $\bsalpha^u_j=1$ if $j\in u$ and $\bsalpha^u_j=0$ if $j\notin u$.

    To apply our results, first we change the anchor point to the box center $(1/2,\dots,1/2)$ and analyze each of the $2^s$ orthants separately. As an example, for the orthant $[1/2,1)^s$ and $\bsy\in [0,1)^s$ we define 
    $$\psi(\bsy)=\Big(\frac{1+y_1}{2},\dots,\frac{1+y_s}{2}\Big)$$
    and $f_\psi(\bsy)=f(\psi(\bsy))$. Clearly $\partial^u f_\psi(\bsy)$ can be defined over $\bsy\in [0,1)^s$ using the chain rule. In order to show $f_\psi\in \mathcal{H}_{\lambda^*,s,2} $ with $\lambda^*\in (0,1/2-A^*)$, we use equation~\eqref{eqn:changelambda} with $\lambda=1$ and $\lambda'=\lambda^*$ to define for $u\subseteq 1{:}s$
    $$D_u^{\lambda^*} f_\psi (\bsy_u)=\Gamma(1-\lambda^*)^{-|u|} \int_{[0,1)^{|u|}} \partial^u f_\psi\Big((\bst_u,\bszero_{u^c})\Big)\prod_{j\in u} (y_j-t_j)_+^{-\lambda^*}\rd \bst_u$$
    where $(\bst_u,\bszero_{u^c})$ is the $s$-dimensional vector whose $j$-th component equals $t_j$ if $j\in u$ and equals zero otherwise. Notice that although $\partial^u f_\psi$ might not belong to $L^1([0,1]^s)$ due to its boundary growth, it is continuous over $[0,1)^s$ and a calculation similar to what we did in Lemma~\ref{lem:increasep} shows equation~\eqref{eqn:Df} is satisfied for $f_\psi(\bsy), \bsy\in [0,1)^s$ with the above $D_u^{\lambda^*} f_\psi, u\subseteq 1{:}s$. To see $D_u^{\lambda^*} f_\psi$ has a finite $L^2-$norm, we let $\delta=1/2-A^*-\lambda^*$ and use equation~\eqref{eqn:Owengrowth} and the chain rule to bound 
    \begin{align*}
       &|D_u^{\lambda^*} f_\psi (\bsy_u)|\\
       \leq &\Gamma(1-\lambda^*)^{-|u|} \int_{[0,1)^{|u|}} \Big(\frac{B}{2^{|u|-(s-|u|)A^*}} \prod_{j\in u} \Big(\frac{1-t_j}{2}\Big)^{-A^*-1}\Big)\prod_{j\in u} (y_j-t_j)_+^{-\lambda^*}\rd \bst_u \\
       = & \frac{2^{sA^*}B}{\Gamma(1-\lambda^*)^{|u|}}\int_{[0,1)^{|u|}}  \prod_{j\in u} (1-t_j)^{-A^*-1} (y_j-t_j)_+^{-1/2+A^*+\delta}\rd \bst_u \\
       \leq & \frac{2^{sA^*}B}{\Gamma(1-\lambda^*)^{|u|}} \Big( \prod_{j\in u} (1-y_j)^{-1/2+\delta/2} \int_{0}^{y_j}  (y_j-t_j)^{-1+\delta/2}\rd t_j \Big)\\
       \leq & \frac{2^{sA^*+|u|}B}{(\delta\Gamma(1-\lambda^*))^{|u|}} \Big( \prod_{j\in u} (1-y_j)\Big)^{-1/2+\delta/2},
    \end{align*}
    where we have used $\delta>0$ because $\lambda^*<1/2-A^*$. Hence, $|D_u^{\lambda^*} f_\psi (\bsy_u)|^2$ is integrable over $[0,1]^{|u|}$ and we prove $f_\psi\in \mathcal{H}_{\lambda^*,s,2} $. After analyzing all of the $2^s$ orthants in a similar manner, we can establish a counterpart of Theorem~\ref{thm:fkboundHp} with $\lambda=\lambda^*$, which eventually leads to $\e((\hat{\mu}^{(r)}_{\infty}-\mu)^2)=O(n^{-2+2A^*+\epsilon})$ for any $\epsilon>0$. We do not elaborate out the details here given a similar analysis has been done in \cite{liu2025randomizedquasimontecarloowens}. It is also straightforward to generalize our analysis to the $\alpha\geq 1$ case with $f^{(\alpha,\dots,\alpha)}$ satisfying the Owen's boundary growth condition.
\end{remark}

\section{Numerical experiment}\label{sec:exp}

Below we will test our algorithm on a list of test functions. Each function will be integrated by the median of QMC estimators with $r=m,E=64$ and two types of randomization, one with completely random designs and one with random linear scrambling. For comparison, the same function will also be integrated by the higher order scrambled digital nets in base 2 of order $1,2$ and $3$ from \cite{dick:2011}. They use Owen's scrambling of order $1,2$ and $3$, respectively. To equalize the computational cost, each higher order digital nets estimator is the average of $2m-1$ independent runs. The pre-designed generating matrices $\mathcal{C}_j$ used in the random linear scrambling and Owen's scrambling are from Sobol's construction with direction numbers from \cite{joe:kuo:2008}. We will estimate the RMSE of each method for $m=1,\dots,16$ by $300$ trials. In the forthcoming plots, medians with completely random designs and random linear scrambling are abbreviated as ``Median CRD" and ``Median RLS", respectively. Scrambled digital nets of order 1,2 and 3 are abbreviated as ``Order 1 DN", ``Order 2 DN" and ``Order 3 DN", respectively.

\subsection{One-dimensional $f$ with discontinuous derivatives}\label{subsec:discont}
We first look at one-dimensional test functions with a point of discontinuity in their derivatives. The experiment uses
$$f_{\alpha^*}(x)=\begin{cases}
    (1-3x)^{\alpha^*} & \text{ if } 0\le x\leq 1/3,\\
   2^{-\alpha^*} ({3x-1})^{\alpha^*} &\text{ if } 1/3< x\le 1,
    \end{cases}$$
for $\alpha^*=1,2$ and $3$. The results are shown in Figure~\ref{fig:alpha}. 

\begin{figure}
    \centering
    \subfloat{\includegraphics[width=0.7\textwidth]{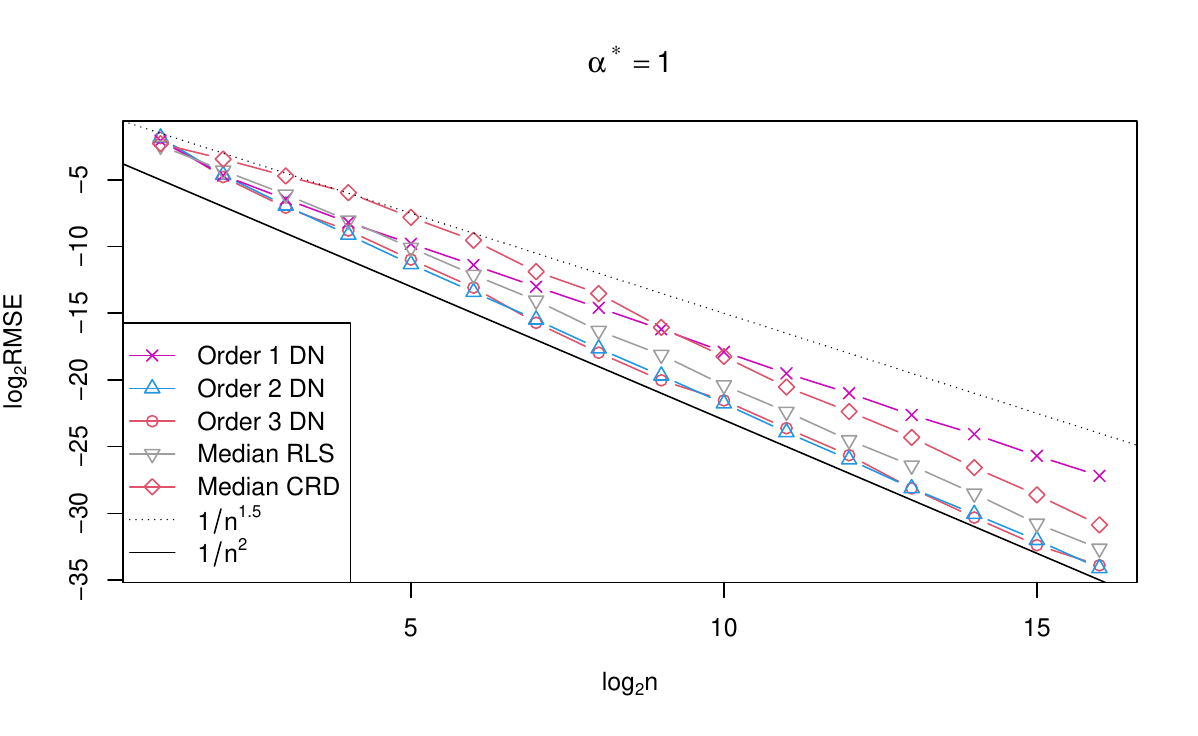}}
    
    \medskip
    \subfloat{\includegraphics[width=0.7\textwidth]{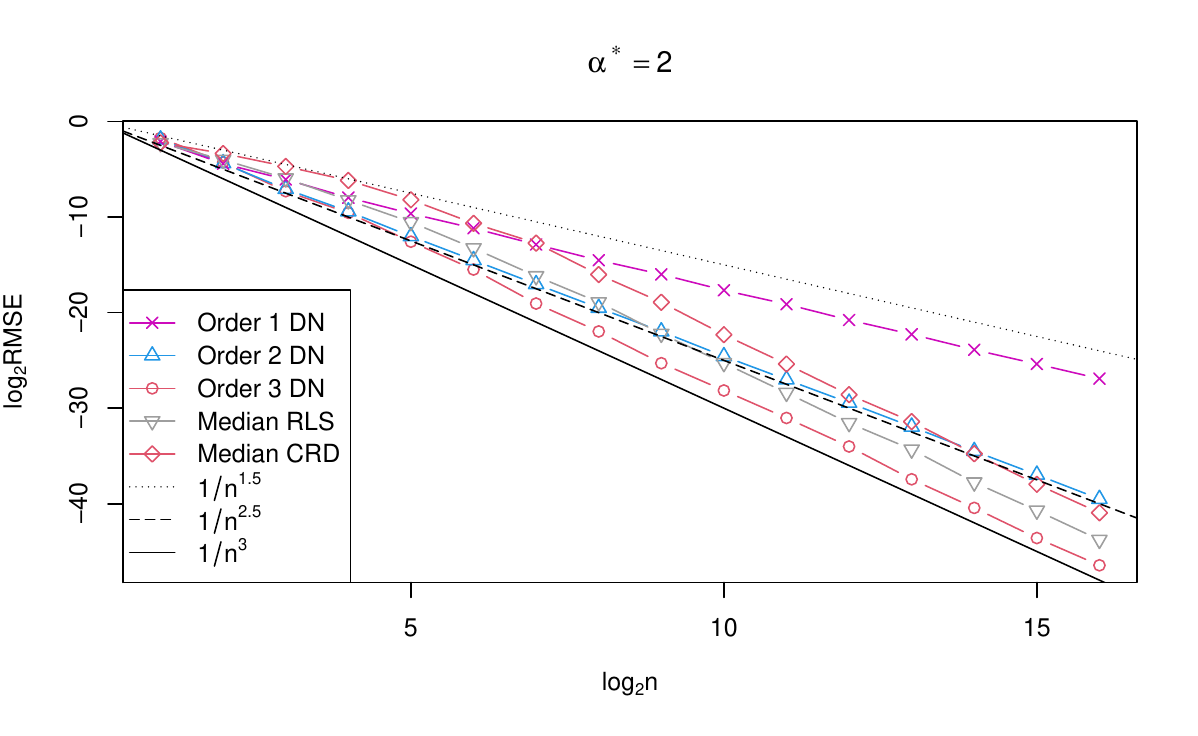}}

    \medskip
    \subfloat{\includegraphics[width=0.7\textwidth]{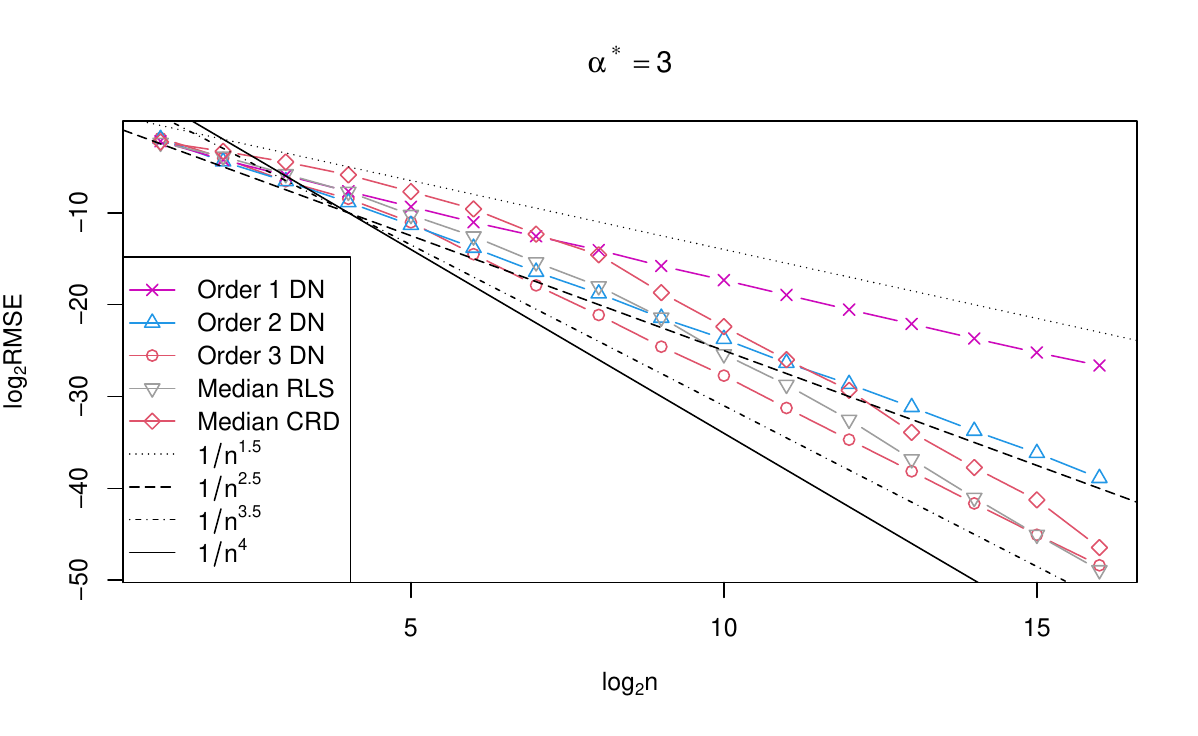}}
    \caption{Comparison of medians with two types of randomization and scrambled digital nets of order $1,2$ and $3$. The test functions are $f_{\alpha^*}$ for $\alpha^*=1$ (top), $\alpha^*=2$ (middle) and $\alpha^*=3$ (bottom). The theory predicts the approximate convergence rate is $O(n^{-\alpha-1})$ for medians with complete random designs ($\diamondsuit$) and medians with random linear scrambling ($\triangledown$), $O(n^{-1.5})$ for digital nets of order $1$ ($\times$), $O(n^{-\min(2.5,\alpha^*+1)})$ for digital nets of order 2 ($\triangle$) and $O(n^{-\min(3.5,\alpha^*+1)})$ for digital nets of order 3 ($\Circle$). }
    \label{fig:alpha}
\end{figure}

Naively, one would think to apply Theorem~\ref{thm:rateHp} with $\alpha=\alpha^*-1$ and $\lambda=1$ because $f^{(\alpha^*)}$ is only defined weakly due to a discontinuity at $x=1/3$. The implied RMSE convergence rate is $O(n^{-\alpha^*-1/2+\epsilon})$. However, the simulation shows the optimal convergence rate is actually $O(n^{-\alpha^*-1+\epsilon})$, which corresponds to Theorem~\ref{thm:rateVlambda} with $\alpha=\alpha^*$ and $\lambda=1/2$. Theorem~\ref{thm:rateVlambda} is applicable here because equation~\eqref{eqn:exactfk} holds despite the discontinuity at $x=1/3$. As a brief proof, we first apply Lemma~\ref{lem:exactfk} with $\alpha=\alpha^*-1$ and get for $k$ with $|\kappa|\ge \alpha^*$
\begin{equation*}
\hat{f}(k)=(-1)^{\alpha^*-1}\int_{[0,1)}f^{(\alpha^*-1)}(x) \walkappa{\kappa\setminus \lceil\kappa\rceil_{1{:}\alpha^*}} (x)(-1)^{\vec{x}(\lceil\kappa\rceil_{\alpha^*})}W_{\lceil\kappa\rceil_{1{:}\alpha^*-1}}(x)\rd x.
\end{equation*} 
Let us denote $\lceil\kappa\rceil_{\alpha^*}$ as $\ell$ for simpilicity.  
Over interval $\mathrm{EI}(\ell-1,a)=[a 2^{-\ell+1},(a+1)2^{-\ell+1})$ with $a\in \ints_{<2^{\ell-1}}$, $\walkappa{\kappa\setminus \lceil\kappa\rceil_{1{:}\alpha^*}}$ is constant and 
\begin{align*}
   &\int_{\mathrm{EI}(\ell-1,a)}f^{(\alpha^*-1)}(x) \walkappa{\kappa\setminus \lceil\kappa\rceil_{1{:}\alpha^*}} (x)(-1)^{\vec{x}(\ell)}W_{\lceil\kappa\rceil_{1{:}\alpha^*-1}}(x)\rd x \\
   =& \walkappa{\kappa\setminus \lceil\kappa\rceil_{1{:}\alpha^*}} (a 2^{-\ell+1}) \int_{\mathrm{EI}(\ell-1,a)}f^{(\alpha^*-1)}(x) \rd W_{\lceil\kappa\rceil_{1{:}\alpha^*}}(x)
\end{align*}
where we have applied the definition of $W_{\lceil\kappa\rceil_{1{:}\alpha^*}}$ from equation~\eqref{eqn:Wkdef}. Because $f^{(\alpha^*)}\in L^\infty([0,1])$, $f^{(\alpha^*-1)}$ is absolutely continuous even over the interval that contains $x=1/3$. Hence we can apply integration by part and get
$$\int_{\mathrm{EI}(\ell-1,a)}f^{(\alpha^*-1)}(x) \rd W_{\lceil\kappa\rceil_{1{:}\alpha^*}}(x)=-\int_{\mathrm{EI}(\ell-1,a)}f^{(\alpha^*)}(x)  W_{\lceil\kappa\rceil_{1{:}\alpha^*}}(x)\rd x$$
where we have used the periodicity of $W_{\lceil\kappa\rceil_{1{:}\alpha^*}}$ to compute $W_{\lceil\kappa\rceil_{1{:}\alpha^*}}(a 2^{-\ell+1})=W_{\lceil\kappa\rceil_{1{:}\alpha^*}}((a+1) 2^{-\ell+1})=W_{\lceil\kappa\rceil_{1{:}\alpha^*}}(0)=0$. Equation~\eqref{eqn:exactfk} with $\alpha=\alpha^*$ then follows once we perform the integration by part over each $\mathrm{EI}(\ell-1,a)$ and add up the resulting integral over $a\in \ints_{<2^{\ell-1}}$. Consequently, as explained in Remark~\ref{rmk:rateVlambda}, we can apply Theorem~\ref{thm:rateVlambda} with $\alpha=\alpha^*$ even though $f^{(\alpha^*)}\notin C([0,1])$.

As we have seen above, it is not always trivial to determine the right function space to analyze the integrand. Practical integrands may be too complicated to analyze or even be defined only in black boxes. A big advantage of taking the median is that no such analysis is needed before we implement the algorithm because it automatically converges at the optimal rate. One may even estimate the optimal rate by carrying out the algorithm and determine which function space to use in a retrospective manner.

\subsection{20-dimensional $f$ with decaying variable importance}

Next we test the performance of the five QMC methods on high-dimensional integrands. The test functions are from Example 4.4 of \cite{Goda2024}:

$$ f_c(\bsx)=\prod_{j=1}^{s}\left[ 1+\frac{1}{\exp(\lceil c\rceil j)}\left(x_j^{c}-\frac{1}{1+c}\right)\right]$$
for $s=20$ and $c=0.5,1.5$ and $2.5$. The results are shown in Figure~\ref{fig:c}.

\begin{figure}
    \centering
    \subfloat{\includegraphics[width=0.7\textwidth]{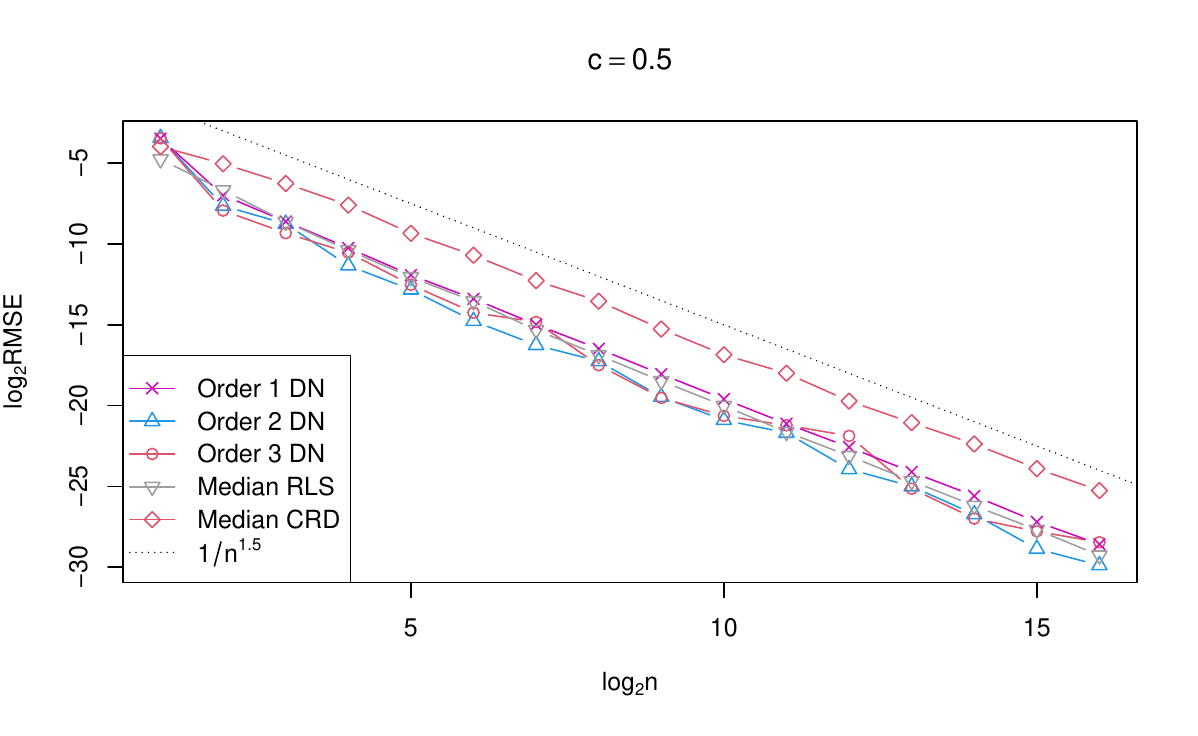}}

    \medskip
    \subfloat{\includegraphics[width=0.7\textwidth]{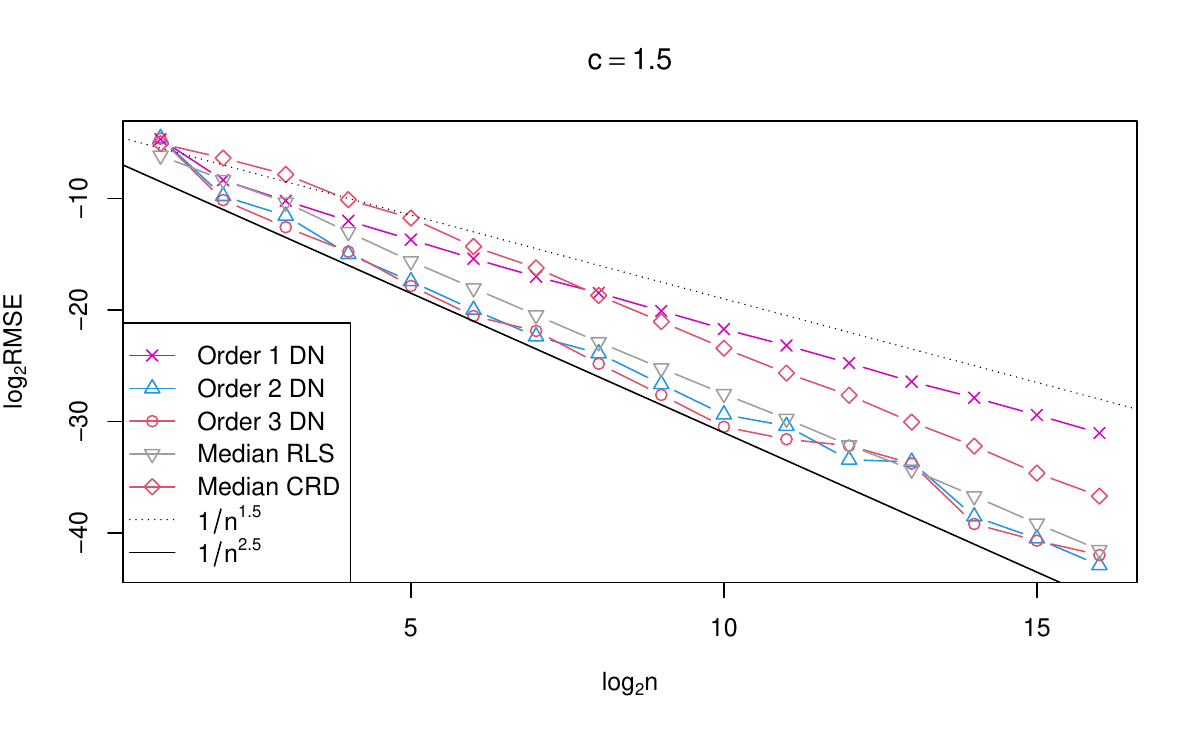}}

    \medskip
    \subfloat{\includegraphics[width=0.7\textwidth]{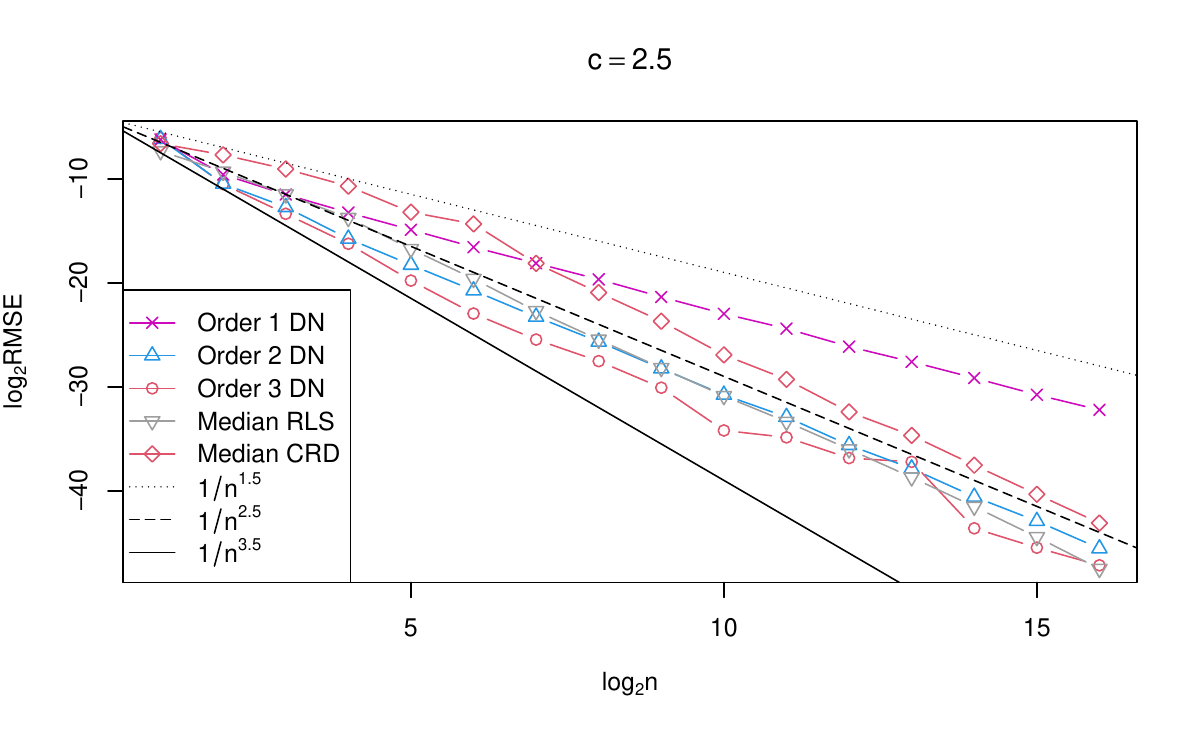}}
    \caption{Comparison of medians with two types of randomization and scrambled digital nets of order $1,2$ and $3$. The test functions are $f_{c}$ for $c=0.5$ (top), $c=1.5$ (middle) and $c=2.5$ (bottom). The theory predicts the approximate convergence rate is $O(n^{-\lceil c\rceil-0.5})$ for medians with complete random designs ($\diamondsuit$) and medians with random linear scrambling ($\triangledown$), $O(n^{-1.5})$ for digital nets of order $1$ ($\times$), $O(n^{-\min(2.5,\lceil c\rceil+0.5)})$ for digital nets of order 2 ($\triangle$) and $O(n^{-\min(3.5,\lceil c\rceil+0.5)})$ for digital nets of order 3 ($\Circle$). }
    \label{fig:c}
\end{figure}

The ANOVA decomposition of $f_c$ can be easily computed because $\prod_{j\in u}(x^c_j-1/(1+c))$ for subsets $u\subseteq 1{:}s$ are mutually $L^2$-orthogonal. The mean dimension of $f_c$, following the definition from \cite{meandim}, is given by 
$$\Big(\prod_{j=1}^s \Big(1+\frac{I_c}{\exp(2\lceil c\rceil j)}\Big)-1\Big)^{-1}\sum_{u\subseteq 1{:}s} |u|\prod_{j\in u} \frac{I_c}{\exp(2\lceil c\rceil j)}$$
where 
$$I_c=\int_{[0,1]}\Big(x^c-\frac{1}{1+c}\Big)^2\rd x=\frac{c^2}{(1+c)^2(1+2c)}.$$
For $s=20$, the above mean dimension is approximately $1+1.04\times 10^{-3}$ for $c=0.5$, $1+3.02\times 10^{-5}$ for $c=1.5$ and $1+5.22\times 10^{-7}$ for $c=2.5$, so $f_c$ is very close to a sum of one-dimensional functions. In particular, the first coordinate $\exp(-\lceil c\rceil )(x^c_1-1/(c+1))$ alone explains more than $80\%$ of the variation.

It is straightforward to verify that $f_c^{(\bsalpha)}\in \mathcal{H}_{1,s,p}$ for $1<p<2$ and $\bsalpha\in \ints^s_{\le \lfloor c\rfloor}$. Corollary~\ref{cor:rateHp} predicts the median has a $O(n^{-\lceil c\rceil-0.5+\epsilon})$  RMSE convergence rate under either randomization. As expected, medians with completely random designs perform the worst among methods with the optimal convergence rate because completely random designs usually generate nets with suboptimal t-parameters growing at $\log m$ rate \cite{pan2024skewnessrandomizedquasimontecarlo}. In contrast, medians with linear scrambling perform comparably to digital nets of order 2 and 3 while enjoying an easier and faster implementation because no digit interlacing is required as in the construction of higher order digital nets. The effect of dimensionality seems to be more salient for smoother integrands, in contrast to our analysis using mean dimension. When $c=2.5$, digital nets of order 3 only performs marginally better than those of order 2 and no methods seem to converge at the $n^{-3.5}$ theoretical rate for the range of $m$ we are testing.

\section{Discussion}\label{sec:disc}

What we have studied in Section~\ref{sec:rate} is far from a complete list of interesting function spaces people have used for the analysis of QMC algorithms. For instance, \cite{gnewuch2024qmcintegrationbasedarbitrary} studies the Haar wavelet spaces, Bessel potential spaces and Besov spaces of mixed dominated smoothness, all of which include the fractional Sobolev space $\mathcal{H}_{\lambda,s,p}$ as a subspace. One may ask whether there exists a universal proof that shows the median $\hat{\mu}^{(r)}_{\infty}$ attains almost the optimal convergence rate for well-conditioned notions of smoothness. It is not clear how this could be done, but it can save our efforts to analyze each smoothness condition individually as we have done in Section~\ref{sec:rate} if such proof exists.

Our analysis is restricted to base-2 digital nets, but the idea should be applicable to digital nets of other prime basis. In particular, \cite{dick:2011} proves Lemma~\ref{lem:Josefbound} in any prime basis and \cite{SUZUKI20161} has a version of Lemma~\ref{lem:exactfk} applicable to odd prime basis. A complete analysis is beyond the scope of this paper and we will leave it for future research.

A more serious limitation of our analysis is the curse of dimensionality. For example, the bound in equation~\eqref{eqn:fkboundHp} contains a constant that grows exponentially with the dimension $s$. Another example is when we bound $(N+\alpha s-1)^{\alpha s-1}$ by $4^{\epsilon N}$ for all $N\geq \lceil m-\alpha s \log_2(m) \rceil$ in the proof of Theorem~\ref{thm:rateVa}. The minimal $m$ for the inequality to hold grows super-linearly in $s$, so the required sample size $n=2^m$ grows super-exponentially in $s$. One way to solve this issue is to focus on a finite-sample analysis and study functions with low effective dimensions as introduced in \cite{cafmowen}. See Section 5 of \cite{superpolymulti} for an example of this approach when the integrand is analytic over $[0,1]^s$. Another solution is to introduce the weighted Sobolev space. A recent paper \cite{Goda2024} uses this idea and gives many dimension-independent bounds on the worst-case error of the median of QMC estimators. Both are interesting directions to explore and we will save it for future research.

Another big topic that deserves further study is how to estimate the error of our median estimate in a reliable way, especially given the bias of $\hat{\mu}^{(r)}_{\infty}$. To simplify the discussion, we again assume the precision $E=\infty$. There have been lots of efforts to construct reliable confidence intervals using randomized QMC \cite{owen2025errorestimationquasimontecarlo}. In our notation, the standard approach is to generate $r$ randomized $\hat{\mu}_\infty$ with the random linear scrambling and compute a $t$-interval based on a normal approximation heuristic. A recent work \cite{gobet2022mean} applies robust estimation techniques to $\hat{\mu}_\infty$ and establishes reliable finite-sample bounds for $\mu$. However, when applying the above methods to function spaces with smoothness parameter $\alpha+\lambda>1$, the resulting confidence intervals tend to be too wide because their lengths are proportional to the standard deviation of $\hat{\mu}_\infty$, which only converges at $n^{-3/2}$ rate due to the influence of outliers \cite{superpolyone}. This phenomenon is also observed in numerical experiments \cite{ci4rqmc}. A remedy is to reduce the variance by applying robust estimation techniques in \cite{gobet2022mean} to higher order scrambled digital net estimates instead, but the problem is again we need a priori knowledge on the function smoothness to choose the order properly.

Ideally, we want to construct confidence intervals that automatically adapt to the function smoothness without a priori knowledge. One possibility is to choose suitable sample quantiles as lower and upper bounds for $\mu$. Just like our median estimate $\hat{\mu}^{(r)}_{\infty}$, any $q$-sample quantile of $\hat{\mu}_\infty$ converges to $\mu$ at almost the optimal RMSE rate provided $q\in (0,1)$ and $r\geq m$. Nevertheless, the target coverage level is only guaranteed in an asymptotic sense, and even that requires both $\Pr(\hat{\mu}_\infty<\mu)$ and $\Pr(\hat{\mu}_\infty>\mu)$ to be bounded away from zero as $m\to\infty$, which is not easy to prove. In a personal communication, Pierre L’Ecuyer proposed a two-stage adaptive procedure. In the first stage, we generate $r$ randomized $\hat{\mu}_\infty$ using either completely random designs or the random linear scrambling. We then find which realization gives the sample median $\hat{\mu}^{(r)}_{\infty}$ and record the generating matrices $C_j$ used by that realization. In the second stage, we generate a new set of randomized $\hat{\mu}'_\infty$ by fixing the generating matrices $C_j$ to be those recorded in the first stage and only randomizing the digital shifts $\vec{D}_j$ by filling them with independent $\dunif\{0,1\}$ entries. One can then construct finite-sample bounds from $\hat{\mu}'_\infty$ by applying techniques in \cite{gobet2022mean}. Intuitively, this method should work well because we expect the variance of $\hat{\mu}'_\infty$ to be comparable to the $\var(\hat{\mu}_\infty\giv \ca^c)$ studied in Lemma~\ref{lem:medianMSE}, which we have shown to decrease at almost an optimal rate. How to coin this intuition into a rigorous theory is an open question leaving for future research.

\section*{Acknowledgments}

The author acknowledges the support of the Austrian Science Fund (FWF) 
Project P34808. For open access purposes, the author has applied a CC BY 
public copyright license to any author accepted manuscript version 
arising from this submission.

\bibliographystyle{plain}
\bibliography{qmc.bib}

\begin{thebibliography}{10}

\bibitem{bahouri2011fourier}
H.~Bahouri.
\newblock {\em Fourier analysis and nonlinear partial differential equations}.
\newblock Springer, 2011.

\bibitem{cafmowen}
R.~E. Caflisch, W.~Morokoff, and A.~B. Owen.
\newblock Valuation of mortgage backed securities using {Brownian} bridges to reduce effective dimension.
\newblock {\em Journal of Computational Finance}, 1:27--46, 1997.

\bibitem{chen2025randomintegrationalgorithmhighdimensional}
L.~Chen, M.~Xu, and H.~Zhang.
\newblock A random integration algorithm for high-dimensional function spaces.
\newblock {\em Preprint}, 2025.

\bibitem{JOSEF:2008}
J.~Dick.
\newblock Koksma–hlawka type inequalities of fractional order.
\newblock {\em Annali di Matematica Pura ed Applicata}, 187:385--403, 07 2008.

\bibitem{dick:2008}
J.~Dick.
\newblock Walsh spaces containing smooth functions and {quasi--Monte Carlo} rules of arbitrary high order.
\newblock {\em SIAM Journal on Numerical Analysis}, 46(3):1519--1553, 2008.

\bibitem{dick:2011}
J.~Dick.
\newblock Higher order scrambled digital nets achieve the optimal rate of the root mean square error for smooth integrands.
\newblock {\em The Annals of Statistics}, 39(3):1372--1398, 2011.

\bibitem{dick2022lattice}
J.~Dick, P.~Kritzer, and F.~Pillichshammer.
\newblock {\em Lattice rules}.
\newblock Springer, 2022.

\bibitem{dick:pill:2010}
J.~Dick and F.~Pillichshammer.
\newblock {\em Digital sequences, discrepancy and quasi-{Monte Carlo} integration}.
\newblock Cambridge University Press, Cambridge, 2010.

\bibitem{Durrett2010}
R.~Durrett.
\newblock {\em Probability: Theory and Examples, 4th Edition}.
\newblock Cambridge University Press, 2010.

\bibitem{gnewuch2024qmcintegrationbasedarbitrary}
M.~Gnewuch, J.~Dick, L.~Markhasin, and W.~Sickel.
\newblock {QMC} integration based on arbitrary (t,m,s)-nets yields optimal convergence rates on several scales of function spaces.
\newblock {\em Preprint}, 2024.

\bibitem{gobet2022mean}
E.~Gobet, M.~Lerasle, and D.~M{\'e}tivier.
\newblock Mean estimation for randomized quasi monte carlo method.
\newblock {\em Hal preprint hal-03631879v2}, 2022.

\bibitem{goda2024simpleuniversalalgorithmhighdimensional}
T.~Goda and D.~Krieg.
\newblock A simple universal algorithm for high-dimensional integration.
\newblock {\em Preprint}, 2024.

\bibitem{Goda2024}
T.~Goda, K.~Suzuki, and M.~Matsumoto.
\newblock A universal median quasi-{M}onte {C}arlo integration.
\newblock {\em SIAM Journal on Numerical Analysis}, 62(1):533--566, 2024.

\bibitem{hein:nova:2002}
S.~Heinrich and E.~Novak.
\newblock Optimal summation and integration by deterministic, randomized, and quantum algorithms.
\newblock In K.~T. Fang, F.~J. Hickernell, and H.~Niederreiter, editors, {\em Monte Carlo and Quasi-Monte Carlo Methods 2000}, pages 50--62. Springer, 2002.

\bibitem{hick:2014}
F.~J. Hickernell.
\newblock {Koksma-Hlawka} inequality.
\newblock {\em Wiley StatsRef: Statistics Reference Online}, 2014.

\bibitem{joe:kuo:2008}
S.~Joe and F.~Y. Kuo.
\newblock Constructing {Sobol'} sequences with better two-dimensional projections.
\newblock {\em SIAM Journal on Scientific Computing}, 30(5):2635--2654, 2008.

\bibitem{ci4rqmc}
P.~L'Ecuyer, M.~K. Nakayama, A.~B. Owen, and B.~Tuffin.
\newblock Confidence intervals for randomized quasi-{M}onte {C}arlo estimators.
\newblock Technical report, hal-04088085, 2023.

\bibitem{meandim}
R.~Liu and A.~B. Owen.
\newblock Estimating mean dimensionality of analysis of variance decompositions.
\newblock {\em Journal of the American Statistical Association}, 101(474):712--721, 2006.

\bibitem{liu2025randomizedquasimontecarloowens}
Yang Liu.
\newblock Randomized quasi-monte carlo and owen's boundary growth condition: A spectral analysis.
\newblock {\em Preprint}, 2025.

\bibitem{mato:1998:2}
J.~Matou\v{s}ek.
\newblock On the {L$^2$}--discrepancy for anchored boxes.
\newblock {\em Journal of Complexity}, 14:527--556, 1998.

\bibitem{niedxing96}
H.~Niederreiter and C.~Xing.
\newblock Low-discrepancy sequences and global function fields with many rational places.
\newblock {\em Finite Fields and Their Applications}, 2:241--273, 1996.

\bibitem{rtms}
A.~B. Owen.
\newblock Randomly permuted $(t,m,s)$-nets and $(t,s)$-sequences.
\newblock In H.~Niederreiter and P.~J.-S. Shiue, editors, {\em Monte Carlo and Quasi-Monte Carlo Methods in Scientific Computing}, pages 299--317, New York, 1995. Springer-Verlag.

\bibitem{haltavoid}
A.~B. Owen.
\newblock Halton sequences avoid the origin.
\newblock {\em SIAM Review}, 48:487--583, 2006.

\bibitem{owen2025errorestimationquasimontecarlo}
A.~B. Owen.
\newblock Error estimation for quasi-{M}onte {C}arlo.
\newblock {\em Preprint}, 2025.

\bibitem{superpolyone}
Z.~Pan and A.~B. Owen.
\newblock Super-polynomial accuracy of one dimensional randomized nets using the median-of-means.
\newblock {\em Mathematics of Computation}, 92(340):805--837, 2023.

\bibitem{pan2024skewnessrandomizedquasimontecarlo}
Z.~Pan and A.~B. Owen.
\newblock Skewness of a randomized quasi-monte carlo estimate.
\newblock {\em Preprint}, 2024.

\bibitem{superpolymulti}
Z.~Pan and A.~B. Owen.
\newblock Super-polynomial accuracy of multidimensional randomized nets using the median-of-means.
\newblock {\em Mathematics of Computation}, 93(349):2265--2289, 2024.

\bibitem{sobol67}
I.~M. Sobol'.
\newblock The distribution of points in a cube and the accurate evaluation of integrals (in {R}ussian).
\newblock {\em Zh. Vychisl. Mat. i Mat. Phys.}, 7:784--802, 1967.

\bibitem{SUZUKI20161}
K.~Suzuki and T.~Yoshiki.
\newblock Formulas for the {W}alsh coefficients of smooth functions and their application to bounds on the {W}alsh coefficients.
\newblock {\em Journal of Approximation Theory}, 205:1--24, 2016.

\end{thebibliography}
\end{document}